\documentclass[12pt,a4paper]{article}
\usepackage[utf8]{inputenc}
\usepackage{amsmath}
\usepackage{amsfonts}
\usepackage{amssymb}
\usepackage{amsthm}
\usepackage{color}
\usepackage{enumerate}

\newcommand{\Vect}{\textnormal{Vect}}

\newcommand{\id}{\textnormal{id}}
\newcommand{\norm}[1]{\left|\left|#1 \right|\right|}

\newcommand{\scal}[1]{\langle #1 \rangle}
\newcommand{\floor}[1]{\lfloor #1 \rfloor}

\newcommand\restr[2]{{\left.\kern-\nulldelimiterspace 
  #1 
  \vphantom{\big|} 
  \right|_{#2} 
  }}

\def\ga{\alpha}\def\gb{\beta}

\def\cH{\mathcal{H}}

\def\cL{\mathcal{L}}

\def\bA{\mathbb{A}}
\def\bN{\mathbb{N}}
\def\bR{\mathbb{R}}

\def\abs#1{\vert #1 \vert} 

\theoremstyle{definition}
\newtheorem{definition}{Definition}[section]

\newtheorem{example}[definition]{Example}
\theoremstyle{plain}
\newtheorem{theorem}[definition]{Theorem}
\newtheorem{proposition}[definition]{Proposition}
\newtheorem{lemma}[definition]{Lemma}
\newtheorem{corollary}[definition]{Corollary}
\theoremstyle{remark}
\newtheorem{remark}[definition]{Remark}

\usepackage{color}
\usepackage{fullpage} 
\usepackage{shuffle}
\usepackage{hyperref}
\usepackage{mathtools}
\usepackage[english]{babel}
\usepackage{comment}
\usepackage{tikz} 
\usepackage{tikz-cd}
\usetikzlibrary{trees}
\usepackage{todonotes}

\usepackage{tikz-cd}

\def\abs#1{\vert #1 \vert}
\def\norm#1{\Vert #1 \Vert}

\def\1{\mathbf{1}}
\def\X{\mathbf{X}}
\def\Z{\mathbf{Z}}
\def\Y{\mathbf{Y}}
\def\R{\mathbb{R}}
\def\W{\mathbf{W}}

\def\qshuffle{\,\widehat{\shuffle}\,}

\newcommand\proj{\mathsf{proj}}
\newcommand\srplus{\boxplus}
\newcommand\srscalar{\boxdot}

\title{Smooth rough paths, their geometry and algebraic renormalization}
\author{Carlo Bellingeri (TU Berlin), Peter K. Friz (TU and WIAS Berlin), \\Sylvie Paycha (Potsdam University), Rosa Preiß (Potsdam University)}
\date{}

\begin{document}
\maketitle

\begin{abstract}
We introduce the class of ``smooth rough paths" and study their main properties.  Working in a smooth setting allows us to discard sewing arguments and focus on algebraic and geometric aspects. Specifically, a Maurer–Cartan perspective is the key to a purely algebraic form of  Lyons extension theorem, the renormalization of rough paths in the spirit of [Bruned, Chevyrev, Friz, Preiß, A rough path perspective on renormalization, J. Funct. Anal. 277(11), 2019] as well as a related notion of ``sum of rough paths". We first develop our ideas in a geometric rough path setting, as this best resonates with recent works on signature varieties, as well  the renormalization of geometric rough paths. We then explore extensions to the quasi-geometric and the more general Hopf algebraic setting. 
\end{abstract}

\tableofcontents

\section{Introduction}

Recent years led to a remarkable convergence of different streams of mathematics. At the center of it is the 
notion of  path $X: [0, T]\to {\mathbb R}^d$ and the ``stack" of its iterated integrals

$$
\X_T :=\mathrm{Sig}(X |_{[0,T]}) := \sum_{n \ge 0}\underset{0 \le u_{1}\leq \ldots
\leq u_{n} \le T }{\idotsint }dX _{u_{1}}\otimes \ldots \otimes
dX _{u_{n}},
$$
commonly called the  signature of $X$ over $[0,T]$, which 
takes values in $T((\mathbb{R}^d))$, the space of tensor series over $\R^d$. It is easy to see that the signature of a path segment actually takes its values in a very special curved subspace of the tensor (series) algebra, $G(\mathbb{R}^d) \subset T((\mathbb{R}^d))$, with a natural group structure.  This construction, which originates in Chen's fundamental work \cite{Chen54} is central to the theory of rough paths and stochastic analysis \cite{lyons1998, lyons2007differential, frizbook}.
Specifically, the path $t \mapsto \X_t :=\mathrm{Sig}(X |_{[0,t]})$ is the canonical rough path lift of $X$, for any sufficiently smooth path $X$ to make the signature well-defined. These ideas have proven useful in a remarkable variety of fields, stretching from machine learning  \cite{chevyrev2016primer} 
to algebraic geometry and renormalization theory. More specifically, we mention:

{\bf Signature varieties.} In  \cite{Amendola2019}, Am\'endola, Sturmfels and one of us study the geometry of signatures tensors. The signatures of a given class of smooth paths parametrize an algebraic variety inside the space of tensors, derived from the free nilpotent Lie group, with surprising analogies with the Veronese variety from algebraic geometry. These signature varieties provide both new tools to investigate paths and new challenging questions about their behaviour. In   \cite{Amendola2019} piecewise linear paths and polynomial paths are investigated. In a later work by Galuppi \cite{galuppi2019rough}, and in the terminology of this paper, the role of classical smooth paths have been replaced by certain {\bf smooth rough paths}, see Definition 2.1. Related recent works on signature varieties include \cite{colmenarejo2020toric,pfeffer2019learning}.

{\bf Rough paths and renormalization.} Rough paths were famously used to solve the singular KPZ stochastic partial differential equation \cite{hairer2013}
and subsequently led to the theory of  regularity structures for general singular SPDEs
\cite{Hairer2014}, with precise correspondences to rough paths highlighted e.g. in \cite[Sec. 13.2.2]{Friz2020course}.  The central topic of {\it renormalization} of singular SPDEs \cite{bruned2019algebraic}  was revisited from a rough path perspective in \cite{Bruned2019}, which notably introduced pre-Lie algebras, and further  inspired progress in the field \cite{bruned2020renormalising},  see also Otto et al. \cite{linares2021structure}.

We cannot possibly expose in full the subtle intertwining of probabilistic, analytic,  geometric and algebraic techniques of the above works, 
but still, sketch the general idea in a simple setting. If $X$ models a realization of noise, such as Brownian sample paths, then the very notion of Stieltjes integration against $dX_u$ is ill-defined (with probability one, such paths are not of locally bounded variation). The stochastic analysis provides probabilistic solutions: Stratonovich integration amounts to work with mollified $X$, followed by taken limits in probability, whereas It\^o integration respects the martingale structure of such processes. Rough path theory (later: regularity structures) understands that these different calculi can be hard-coded  imposing the first $N$-iterated integrals of $X$. The resulting object, an enhancement of $X$, is called a rough path (resp. model in the context of regularity structures). 

There are situations - notably the KPZ equation - when the Stratonovich solution diverges, whereas the correct and desired object is given by the It\^o solution. From a rough paths perspective, this amounts to adjusting (``renormalize'') the higher levels of the aforementioned enhancement. Doing so in an algebraically consistent way is a highly non-trivial task, and was achieved in a rough path resp. regularity structure set in the afore-mentioned works. In particular, \cite{bruned2019algebraic}  develops the algebraic theory of renormalization exclusively for  {\em smooth models}, to which our study of {\em smooth rough paths} is aligned.

\medskip
In this paper, we look at the indefinite signature  path $t \mapsto \X_t :=\mathrm{Sig}(X |_{[0,t]})$, also known as canonical rough path lift, of a $\bR^d$-valued smooth path $X$, as the solution of the linear differential equation
\begin{equation} \label{eq:MCex}
\dot{\X}_t = \X_t \otimes \dot{X}_t, \qquad \X_0 = \1 \in \mathcal{G}, 
    \end{equation}
with, as is well-known (see e.g. \cite[Sec. 2.1.1]{lyons1998} or \cite[Ch. 7]{frizbook}) $\mathcal{G} := G( \bR^d ) = \exp_{\otimes}(\mathfrak{g}) \subset T((\mathbb{R}^d))$, where  $\mathfrak{g} := \cL (( \bR^d)) = \bR^d \oplus [\bR^d,\bR^d] \oplus ... $ denotes the space of Lie series. The same construction in the quotient (or level-$N$ truncated) algebra 
$$
(T^N(\mathbb{R}^d),\otimes_N, \1)
$$
yields a finite-dimensional Lie group $\mathcal{G}^N := G^N( \bR^d )= \exp_{\otimes_N}(\mathfrak{g}^N)
\subset T^N ( \R^d)$ with Lie algebra $\mathfrak{g}^N := \cL^N ( \bR^d)$, the Lie polynomials of degree less equal $N$.
This yields an increasing family of Lie algebras (resp. groups) with inclusion map $i$ (resp. $j$). 

 \begin{center} 
\begin{tikzcd}
\cdots \arrow[r, "j"] 
& \mathcal{G}^N 
\arrow[r, "j"] \arrow[d]
& \mathcal{G}^{N+1}
\arrow[r, "j"] \arrow[d, ] & \cdots \arrow[r, "j"] 
& \mathcal{G} \arrow[d]
\\
\cdots \arrow[r, "i"] 
& \mathfrak{g}^N 
\arrow[r, "i"] 
& \mathfrak{g}^{N+1} \arrow[r, "i"] &
\cdots \arrow[r, "i"] & \mathfrak{g}
\end{tikzcd}
\end{center}
Cartan's classical development of a smooth $\mathfrak{g}^N$-valued path $\mathfrak{y}$ is precisely given by solving the differential equation
$$
\dot{\X}_t = \X_t \otimes_N \mathfrak{y} (t), \qquad \X_0 = \1 \in \mathcal{G}^N. 
$$
Basic results on linear differential equations in finite dimensions guarantee a unique and global solution. 
This gives a well-defined ``projective'' way to solve for $\dot{\X}_t = \X_t \otimes \mathfrak{y} (t), \X_0 = 1 \in \mathcal{G}$, 
for any $\mathfrak{g}$-valued path $\mathfrak{y}$. This is a precise generalization of \eqref{eq:MCex} and leads us to the class of {\em smooth rough paths}\footnote{(Warning.) We insist again that smooth rough paths are much richer objects than {\em canonical lifts of smooth paths}, unfortunately also called smooth rough paths in the earlier stages of the theory.
To wit, the {\em pure area rough path}, familiar in rough path theory and seen later on in the text, is a perfect example of a smooth rough path.}. In infinite-dimensional Lie group theory, solvability of such equations has led to the notion of {\em regular Lie group}  \cite{Milnor}; although this is of no concern to us and we refer to  \cite{Bogfjellmo2016,Bogfjellmo2018}
 for the subtleties of infinite groups like $\mathcal{G}$.

Conversely, every smooth $\mathcal{G}^N$-valued path $\X_t$ is the Cartan development of  
$$
\mathfrak{y} (t) :=  \langle\omega_{\X_t}, \dot{\X}_t \rangle = \X_t^{-1} \otimes \dot{\X}_t = \partial_h |_{h=0} {\X_{t,t+h}} =: \dot{\X}_{t,t} \in \mathfrak{g}^{N}.
$$
Here $\omega_{\mathbf{x}}$ is the $\mathfrak{g}^N$-valued (left invariant) {\it Maurer--Cartan}  form, given at $\mathbf{x} \in G^N(\mathbb{R}^d)$, 
by
\begin{equation}\label{eq:MaurerCartan}
    \omega_\mathbf{x} := \mathbf{x}^{-1} \otimes d\mathbf{x}.
\end{equation} 
which can be viewed as left logarithmic derivative of the identity map of  $G^N(\mathbb{R}^d)$. The reason we encounter the left invariant Maurer--Cartan form, rather than its right invariant counterpart ($(d\mathbf{x}) \otimes \mathbf{x}^{-1}$) can be traced back to the order of interated integrals in the definition of the signatures, i.e. $u_1 \le \dots \le u_n$ rather than $u_n \le \dots \le u_1$.  We recall also that the Maurer-Cartan  form has appeared in previous works of rough paths on manifolds \cite{cass2015constrained}, as well as signature based shape analysis \cite{celledoni2019signatures}.  We shall see in \S \ref{sec:alg_ren_geo} and \S \ref{sec:alg_ren_qgeo}, that renormalization of rough differential equations, in the spirit of \cite{Bruned2019}, can much benefit from this geometric view point. (For earlier use in the context of renormalization see also \cite{CQRV}) In the context of renormalization of rough differential equations however, its use appears to be new.

The geometry of $\mathcal{G}$ encodes validity of a chain rule, equivalently expressed in terms of shuffle identities, that in turn exhibits $\mathcal{G}$ as a character group of the shuffle Hopf algebra. This suggests correctly that the Maurer-Cartan perspective is not restrictive to $G( \bR^d )$-valued (``geometric'') rough paths, but valid for ``general'' rough paths, in the sense of \cite{Tapia2018}, with values in the character groups of a general graded Hopf algebra.
However, too much generality does not allow for some of the concrete applications we have in mind, notably an understanding of differential equations driven by rough paths and their renormalization theory. We thus commence in Chapter 2 with smooth instances of {\em geometric rough paths} (in short: grp), which can be thought of as a multidimensional path enhanced with iterated integrals with classical integration by parts (shuffle) relations, together with suitable analytic conditions.  A typical (non-smooth) example is given by Brownian motion with iterated integrals in the Stratonovich sense,  see e.g.\cite[Sec. 2.2] {Friz2020course} for precise definitions, see also 
\cite{lyons1998,lyons2007differential,frizbook, hairer2015geometric,Friz2020course}. (We will review what we need in the main text below, cf. Definitions \ref{def:smooth-geom} and  \ref{defn_model_smooth}.) Another aspect concerns the {\em sub-Riemannian} structure of $G^N( \bR^d )$, the state space of (level-$N$) geometric rough paths, see e.g. in \cite[Remark 7.43]{frizbook} or \cite{friz2016geometric}.

The indefinite signature of a $\bR^d$-valued path always stays tangent to the left-invariant vector fields generated by the $d$ coordinate vector fields. In sub-Riemannian geometry, such paths are called {\em horizontal}. The study of smooth geometric rough paths is effectively the study of (possibly non-horizontal) smooth paths on $\mathcal{G}$. It is classical in geometric rough path theory to equip this space with the Carnot--Caratheodory metric. A generic smooth geometric rough path then has infinite length and are thus a genuine and interesting example of rough paths. We then continue to extend smooth rough paths in a more general setting:

{\bf Quasi-geometric theory:}
A {\em quasi-geometric rough path} should then be thought of as a multidimensional path enhanced with iterated integrals with classical integration by parts (shuffle) relation replaced by a
generalized integration by parts rule known as quasi-shuffle. A typical example is given by Brownian motion with iterated integrals in the It\^o sense. This structure also arises naturally when dealing with discrete sums (cf. \cite{Tapia20} for signature sums) or piecewise constant paths, as well as L\'evy processes \cite{curry14}, general semimartingales and finally rough path analysis \cite{bruned2020renormalising,Bel2020a}. We also note unpublished presentations on quasi-rough paths by D. Kelly, who first introduced the concept, following his work \cite{kelly2012ito,hairer2015geometric}. The very influential article \cite{hairer2015geometric} focused on the interplay between geometric and {\it branched rough paths}, introduced in \cite{gub10}, not central to this work. We note that quasi-shuffle structures have emerged independently in renormalization theory, see e.g. \cite{Kreimer,MP,manchon2010nested,Clavier2020} and the references therein.

{\bf Hopf algebra constructions:} We finally revisit the previous constructions from a general Hopf algebra construction. In view of \cite{Bruned2019} we do not single out the case of branched rough paths. Our graded approach to treat first geometric (Chapter 2), then quasi-geometric (Chapter 3) and finally the Hopf algebra case (Chapter 4) is a choice we made for essentially two reasons: (i) The material of Chapter 2 remains accessible to readers with a minimum on prerequisites and is also the setting that is most used in the context of signatures, including the recent developments in algebraic-geometric. (ii) Not every results obtained in the (quasi)geometric setting has a precise counter-part in the Hopf algebraic generality. For instance, as already observed in \cite{Bruned2019} in the context of branched rough paths, one loses certain uniqueness properties of renormalization operators when passing to more general structures.

\medskip

Let us list the main contributions of this work with some detailed pointers to the main text. 

\begin{itemize} 
\item  With Definitions 2.1, also 3.8, 4.1 we introduce the class of smooth rough paths in their respective setting and show in Theorems 2.8, also 3.10, 4.2 that level-$N$ smooth rough paths have a lift whose uniqueness hinges on some algebraic/geometric minimality. (This is in contrast to the classical Lyons lift of level-$N$ rough paths, where uniqueness depends on analytical conditions, see e.g. \cite[Ch. 9]{frizbook}.) We note that a related minimality condition appeared in the context of rough paths with jumps \cite{friz2017general}. 
\item An interesting insight is then that smooth rough paths, the resulting space of which is by nature non-linear, can be given a canonical linear structure. (This should be contrasted with adhoc linearizations based on Lyons--Victoir extension, see e.g. \cite[Ex. 2.4]{Friz2020course} and especially \cite{Tapia2018}.)
\item We finally revisit differential equations driven by our classes of rough paths, followed by their renormalization as initiated in \cite{Bruned2019}. Specifically, in Theorem 2.26 we highlight the role of smooth rough paths in the argument. Our subsequent extensions in Section 3 and 4 complements (and differ from) existing results, notably \cite{Bruned2019}, with a sole focus on geometric and branched structures, and related works \cite{bruned2020renormalising, Bruned2020} that involve/pass through branched constructions.
\end{itemize}

\bigskip
{\bf Acknowledgment.} 
CB, PKF and SP were supported in part by DFG Research Unit FOR2402. PKF and RP were supported in part by the European Research Council (ERC) under the European Union’s Horizon 2020 research and innovation program (grant agreement No. 683164). RP would like to thank Terry Lyons for a discussion that led to fundamental ideas for Section \ref{sec:srplus_geom}, Corollaries \ref{cor_time_transl}, \ref{cor:srplus_quasi}, \ref{cor:srplus_quasi_trans} and Section \ref{sec:srplus}.
\bigskip
\bigskip

\section{Smooth geometric rough paths}\label{sec:sgrp}

\subsection{Definitions and fundamental properties}\label{subsec:defns}

Let $(T((\bR^d)), +, \otimes)$ denote the algebra of tensor series over $\bR^d$, equipped with the standard basis $ e_1,..,e_d$. Elements of $T((\bR^d))$ are of the form
$$
\mathbf{x} = \sum \mathbf{x}^w e_w\,, 
$$
with summation over all words $w=\ell_1 \cdots\ell_n$ with letters $\ell_j \in \{1,...,d \}$, scalars $\mathbf{x}^w$ and $e_w := e_{l_1} \otimes \dots \otimes e_{l_n}$. The summation includes also $\mathbf{1}$, the empty word. There is a natural pairing of $T((\bR^d))$ with $T(\bR^d)$, the space of tensor polynomial linearly spanned by the $e_w$, so that
\begin{equation}\label{pairing}
\langle \mathbf{x}, w \rangle := \langle \mathbf{x}, e_w \rangle = \mathbf{x}^w,
\end{equation}
and the same pairing applies to the truncated spaces $T^N(\bR^d)$, consisting  of tensor polynomials of degree at most $N\in \mathbb{N}$, spanned by pure tensors $e_w$ whose word $w$ has length $|w|\le N$. We denote the the canonical projection onto $T^N(\bR^d)$ as $\proj_N\colon T((\bR^d))\to T^N(\bR^d)$. Equivalently, we can introduce $T^N(\bR^d)$ as a quotient algebra. Indeed, introducing the ideal $T^{>N}((\bR^d)) := \bigoplus^{\infty }_{n> N}(\bR^d)^{\otimes n}$ one immediately sees 
$$
   T^N(\bR^d) \cong T\left( \left( \bR^d\right) \right)  / T^{>N}((\bR^d)) .
$$
We write $\otimes_N$ for the induced ``truncated tensor product''.

Using the identification between words and tensors, we denote by $\shuffle$ the shuffle product on words $\shuffle\colon T(\bR^d)\times T(\bR^d)\to T(\bR^d)$, defined by $w\shuffle \mathbf{1}= \mathbf{1}\shuffle w = w$  for any word $w$ and the recursive definition
\begin{equation}\label{shuffle_recursive}
w \,i \shuffle v\,j= (w \shuffle v\,j)\, i + (v\,i \shuffle w)\, j \,,
\end{equation}
for any couple of words $w,v$ and letters $i,j\in \{1,...,d \}$. The product $\shuffle$ induces a commutative algebra\footnote{In fact, a Hopf algebra on $T(\bR^d)$ together with the deconcatenation coproduct, see e.g.  \cite{reutenauer1993free}, though this will not play a role in this section.} on $T \left(\bR^d\right)$.

There are several Lie algebras we want to look at which stem from the algebra $(T\left( \left(\bR^d\right) \right),\otimes) $; we use the notation $\mathcal{L}\left( \bR^d\right) $ for
the Lie algebra generated by the letters $\bR^d$ (the space of Lie polynomials), $\mathcal{L}^{N}\left( \bR^d\right) $ for the projection of {$\mathcal{L}\left( \bR^d\right) $} into $T^N (\bR^d)$, 
and $\mathcal{L}\left( \left( \bR^d\right)
\right) \subset T((\bR^d))$ 
the Lie series. We define the tensor and truncated tensor exponential 
$\exp_{\otimes}\colon T^{>0}((\mathbb{R}^d)) \to T((\mathbb{R}^d))$,  $\exp_{\otimes_N}\colon T^N(\mathbb{R}^d)\cap T^{>0}((\mathbb{R}^d))\to T^N(\mathbb{R}^d)$ given respectively by  
\begin{equation}\label{exp_tensor}
\exp_{\otimes}\mathbf{x}= \sum_{n \ge  0}\frac{\mathbf{x}^{\otimes n}}{n!}\,, \quad \exp_{\otimes_N}\mathbf{x}= \sum_{n=0}^N\frac{\mathbf{x}^{\otimes_N n}}{n!}.
\end{equation}
Then it is well known (\cite{reutenauer1993free,lyons2007differential}) that $G(\bR^d)= \exp_\otimes \cL (( \bR^d ))$ is a group with  operation $\otimes$ which satisfies 
\begin{equation} \label{equ:recallG}
G(\bR^d) = \{ \mathbf{x} \in T((\bR^d)): \langle\mathbf{x},\mathbf{1}\rangle=1,\,\,\langle \mathbf{x},v \shuffle w\rangle = \langle \mathbf{x},v\rangle\langle\mathbf{x},w\rangle
\text{ for all words $w,v$} \}\,,
\end{equation}
Similar results hold  for $G^N(\mathbb{R}^d)$, now in terms of words with joint length $\le N$ and as the image of $\mathcal{L}^{N}\left( \bR^d\right)$ under $\exp_{\otimes_N}$ respectively. We note that $(G^N(\mathbb{R}^d), \otimes_N)$ is a bona fide finite-dimensional Lie group with Lie algebra $\mathcal{L}^{N}\left( \bR^d\right)$.
\medskip

The following definitions are standard (as e.g. found in \cite{hairer2015geometric}), but with analytic H\"older / variation type conditions replaced by a smoothness assumptions.

\begin{definition}\label{def:smooth-geom}
We call level-$N$ {\bf smooth geometric rough path} (in short: $N$-{\bf sgrp}) over $\bR^d$ any non-zero path { $\mathbf{X}: [0,T] \to T^N(\bR^d) $ } such that 
\begin{itemize}
    \item[(a.i)] The shuffle relation holds for all times $t \in [0,T]$  
\begin{equation}\label{eq:shufflerel}
\langle\X_{t},v \shuffle w\rangle = \langle \X_{t},v\rangle\langle\X_{t},w\rangle
\end{equation} for all words with joint length $|v|+|w| \le N$. 
\item[(a.ii)] For every word of length $|w| \le N$, the map $t \mapsto\langle\X_{t},w\rangle$
is smooth. We write 
\[\dot{\X}_t = \sum_{|w|\le N} \langle\dot{\X}_{t},w\rangle e_w\] for the derivative of $\X$.
\end{itemize}
By {\bf smooth geometric rough path} (in short: {\bf sgrp})  we mean a path with values in  $T((\bR^d))$ with all defining ``$\le N$'' restrictions on the word's length omitted. 
\end{definition}
\begin{remark} 
By Proposition 2.14 below $N$-sgrp are smooth $G^N(\mathbb{R}^d)$-valued paths, thus (with respect to the appropriate Carnot-Caratheodory metric \cite{frizbook}) genuine $1/N$-H\"older regular rough paths, which justifies our terminology. Similarly, sgrp's are nothing but smooth $G(\mathbb{R}^d)$-valued paths, provided $G(\mathbb{R}^d)$ is equipped with a suitable ``weak'' differential structure to make it a topological Lie group, see \cite{Bogfjellmo2018}. 

\end{remark}
The shuffle relation applied to empty words and the demand that $\X$ is non-zero imply together with continuity in $t$ that $\langle\X_t,\mathbf{1}\rangle=1$ for all $t\in[0,T]$, hence $\X$ is (similar to formal power series) invertible with respect to $\otimes_{N}$. Every $N$-sgrp gives then rise to increments, $(s,t) \mapsto \X_s^{-1} \otimes_{N} \X_t$. This motivates the following definition.
\begin{definition}\label{defn_model_smooth}
We call level-$N$ {\bf smooth geometric rough model} (in short: $N$-{\bf sgrm}) over $\bR^d$ any non-zero map { $\X: [0,T]^2 \to T^N(\bR^d)$}    such that
\begin{itemize}
\item[(b.i)] the  shuffle relation \eqref{eq:shufflerel} holds with $\X_t$ replaced by $\X_{s,t}$, any $s,t$. 
\item[(b.ii)]  Chen's relation holds, by which we mean
\begin{equation}\label{eq:Chensrel}
\X_{su} \otimes_N \X_{ut}=\X_{s,t}\,
\end{equation} 
for any $s,u,t\in[0,T]$. 
\item[(b.iii)] For every word of length $|w| \le N$, the map $\mapsto \langle\X_{s,t},w\rangle$
is smooth, for one (equivalently: all) base point(s) $s \in [0,T]$.
\end{itemize}
By {\bf smooth geometric rough model} (in short {\bf sgrm})  we mean a map with values in $T((\bR^d))$, with all ``$\le N$'' quantifiers omitted and equation \eqref{eq:Chensrel} with $\otimes_N$ replaced by $\otimes$. 
\end{definition}
\begin{remark}
The terminology ``model'' in  is consistent with Hairer's regularity structures. More specifically, given a $N$-sgrm $\X = \X_{s,t}$ we
have the map 
$$
    s \mapsto \{ T^N(\R^d) \ni \mathbf{u} \mapsto \langle \X_{s,\cdot}, \mathbf{u} \rangle \}
$$
which together with Chen's relation yields precisely a model in the sense of regularity structures. See e.g. \cite[Sec. 13.2.2]{Friz2020course}, \cite[Thm. 5.15]{Preiss16} and \cite[Proposition~48]{Bruned2019}.
Our use of the adjective ``smooth'' is also consistent with the notion of \textbf{smooth model}, used by Hairer and coworkers and central to the algebraic renormalization theory of \cite{bruned2019algebraic}. 
\end{remark}

It follows from \eqref{eq:shufflerel} resp.\ (b.i) and the non-zero demand that $N$-smooth geometric rough paths and models really take values in $G^N(\R^d)$; similarly for (tensor series) smooth rough paths and models and $G(\R^d)$. Clearly, every level-$N$ smooth geometric rough path $\mathbf{X}:[0,T]\to T^N(\bR^d)$ induces a level-$N$ smooth geometric rough model $\mathbf{X}: {[0,T]^2} \to T^N(\bR^d) $ in the above sense, by considering the increments 
\begin{equation}\label{eq:map_to_path}
    \X_{s,t} := \X_{s}^{-1} \otimes_N \X_t.
\end{equation}
Conversely, every level-$N$ smooth geometric rough model defines a level-$N$ smooth geometric rough path $\X_t := \X_{0,t}$ so that $N$-sgrp and $N$-sgrm are equivalent modulo a starting point in $G^{N} (\R^d)$.

\begin{definition}\label{def_extension}
A sgrp $\X$ 
is called {\bf extension} of some $N$-sgrp $\Y$ if 
$$
\langle \X_t, w \rangle = \langle \Y_t, w \rangle \qquad \text{for all $t$ and whenever $|w| \le N$;} 
$$
if this holds for a $N'$-sgrp $\X$, with $N   < N' < \infty$, we call it $N'$-{\bf extension} of $\Y$. We adopt also the same denomination if $\Y$ is a $N$-sgrm and $\X$ is a sgrm ($N'$-sgrm) and one has the same relation above for any $s,t\in [0,T]$.
\end{definition}

When $d=1$, the situation is trivial, and sgrp's are in one-to-one correspondence with smooth scalar paths: An arbitrary scalar path $Y$ with initial point $Y_0=0$ has a  (unique) extension   given by
$$
\X_t = \mathbf{1} + Y_t \,e_1 +  \frac{(Y_t)^2}{2!}\, e_{11}+ ... =\exp_{\otimes}{Y_t} \in T((\bR))\,.
$$
Conversely, every sgrp with $\X_0=\mathbf{1}$ must be of this form, as a consequence of the shuffle relations. When $d>1$ such extensions are never unique. For instance, given the two basis vectors $e_1,e_2\in\mathbb{R}^2$ and denoting by $\mathbf{0} $ the common zero of $(\mathbb{R}^d)^{\otimes n} $, $n\geq 1$ we see that $t \mapsto (1,\mathbf{0},\mathbf{0})\in T^2(\bR^2)$ and $t \mapsto (1,\mathbf{0}, t[e_1,e_2]) \in T^2(\bR^2)$ are both $2$-sgrps over $\bR$, and hence level-$2$ extensions of the trivial $1$-sgrp $t \mapsto (1,\mathbf{0})$. 

In classical rough path analysis \cite{lyons1998,frizbook},  a graded $p$-variation (or H\"older) condition  is enforced, which guarantees a unique extension. We give here a novel algebraic condition that enforces uniqueness, somewhat similar in spirit to the minimal jump extension of cadlag rough paths in \cite{friz2017general}.
This condition is motivated by the following result.
\begin{proposition} \label{prop:lie1}
Given a sgrm $\X$ over $\bR^d$, for all times $s$ one has,
\begin{equation}\label{prop:diagonal_derivative}
\dot{\X}_{s,s}:= 
\partial_t |_{t=s} \X_{s,t} \in {{\mathcal L}}((\bR^d))\,,
\end{equation}
The analogue statement  { holds}   for $N$-sgrm's, with $ {{\mathcal L}}((\bR^d))$ replaced by its truncation $\cL^N(\bR^d)$. We call the path in \eqref{prop:diagonal_derivative} the \textbf{diagonal derivative} of $\X$.
\end{proposition}
\begin{proof} A geometric proof is not difficult. We define the $G^N(\R^d)$-valued path $\X_t := \X_{0,t}$ so that $\dot{\X}_{s,s}:= \X_s^{-1} \otimes \dot{\X}_s$
which is exactly the  Maurer--Cartan form at $\X_s$, evaluated at the tangent vector $\dot{\X}_s$. This results in an element in the Lie algebra, which here is identified with $\cL^N(\bR^d)$. Since $N$ is abitrary, this also implies the first claim.
\end{proof}

\begin{remark}\label{rk_N2}
 Let us illustrate the previous Proposition in case $N=2$. In view of the defining shuffle relation of sgrm's, applied to single letter words $i$ and $j$,
\[\langle\X_{s,t},i \rangle\langle\X_{s,t},j  \rangle= \langle \X_{s,t},ij\rangle+\langle\X_{s,t},ji\rangle. \]
 Dividing this identity by $t-s>0$, followed by sending $t\downarrow s$, one has immediately
 \begin{equation}\label{diagonal_derivative}
     \langle \dot{\X}_{s,s},ij\rangle+\langle\dot{\X}_{s,s},ji\rangle=0\,
 \end{equation}
so that the second tensor level of $\dot{\X}_{s,s}$ is anti-symmetric, hence $\dot{\X}_{s,s} \in \R^d \oplus [\R^d,\R^d]$.
\end{remark}

\begin{theorem}[Fundamental Theorem of sgrm]\label{fundthmSGRP}
{Given an $N$-sgrm  $\Y$ for some $N \in \bN$}, there exists exactly one sgrm extension $\X$  of $\Y$ 
which is minimal in the sense that for all $s\in[0,T]$ one has
$$
\dot{\X}_{s,s} \in \cL^N (\bR^d) \subset \cL ((\bR^d)).
$$
This unique choice then in fact satisfies $\dot{\X}_{s,s}=\dot{\Y}_{s,s}$.
We call $\mathrm{MinExt(\Y)} := \X$ the {\bf minimal extension} of $\Y$ and also $\mathrm{MinExt^{N'}(Y)} := \proj_{N'} \X$, for $N' > N$, the $N'$-minimal extension of $\Y$. For a fixed interval $[s,t] \subset [0,T]$ it holds that $\X_{s,t}$ only depends on $\{ \Y|_{[u,v]}: s\le u \le v \le t \}$ and we introduce the {\bf signature} of $\Y$ on $[s,t]$ by 
$$\mathrm{Sig}(\Y|_{[s,t]}) := \X_{s,t} \in G(\bR^d).$$
\end{theorem}

\begin{proof} (Existence) 
Thanks to the previous proposition
$$
 \mathfrak{y} (t) := \dot{\Y}_{t,t} \in \cL^N (\bR^d) 
$$
defines a smooth path with values in the Lie algebra 
$\cL^N (\bR^d)$. Its Cartan development into $G(\R^d)$ amounts to solve for $\dot{\X}_t = \X_t \otimes \mathfrak{y} (t), \X_0 = \mathbf{1} \in G(\R^d)$. It is enough to do this in finite dimensions, say in $G^{N'}(\R^d)$ for arbitrary $N'>N$, in which case such differential equations have a unique and global solution. Indeed, existence of a unique local solution is clear from ODE theory, whereas non-explosion is a consequence of linearity of this equation. See also a much more general reference \cite{iserles1999solution}. Since the natural projections from $G^{N'+1}(\R^d) \to G^{N'}(\R^d)$ are Lie group morphisms, we obtain a consistent family of $N'$-extensions which defines $t\mapsto\X_t$ in the projective limit. The minimal extension is then given by $\X_{s,t} = \X_s^{-1} \otimes \X_t$, where we note that
$$
   \dot{\X}_{t,t} = \X_t^{-1} \otimes  \dot{\X}_t = \mathfrak{y} (t) \in \cL^N (\bR^d) \,.
$$

(Uniqueness) It is sufficient to consider $N'=N+1$. In this case two level-$N'$ extensions $\X,\bar{\X}$, differ by an element in the center of $G^{N'}(\R^d)$, so that $\Psi_{s,t} := \langle \X_{s,t} -\bar{\X}_{s,t}, w \rangle$ is additive for every word $|w| \le N+1$. We need to show that $\Psi \equiv 0$. By assumption, $\Psi$ vanishes on words of length $|w| \le N$, so we can assume $|w| = N'=N+1$. 
With $\Psi_t := \Psi_{0,t}$ note that $\Psi_{s,t} = \Psi_t - \Psi_s$. Write 
$\X_t = \X_{0,t}$, and similarly for $\bar{\X}$, and also  $\sim$ for equality in the limit $t\downarrow s$. Then
$$
   \dot{\Psi}_s 
    \sim
    \frac{\Psi_{s,t}}{(t-s)}
    =
    \frac{\langle \X^{-1}_s\otimes (\X_t - \X_s) -\bar{\X}^{-1}_{s} \otimes (\bar{\X}_t -\bar{\X}_s ), w \rangle }{(t-s)}
    \sim
    \langle \X^{-1}_s\otimes \dot{\X}_s  -\bar{\X}^{-1}_{s} \otimes \dot{\bar{\X}}_s , w \rangle 
$$
and in view of the minimality assumption of order $N$ of both $\X,\bar{\X}$, with the condition $|w| = N+1$, we see that $\dot \Psi$ is zero, hence $\Psi \equiv \Psi_0 =0$ which concludes the argument for uniqueness.
\end{proof}
\begin{remark}[Computing the minimal extension] 
Note that the existence part of this proof is constructive and gives the minimal extension via solving a linear differential equation. To make this explicit in case of a $N$-sgrp $Y$, it suffices to solve
$$
      \dot{\X} = \X \otimes (\Y^{-1} \otimes_N \dot{\Y}), \qquad  {\X}_0 = \mathbf{1} \in T((\bR^d))\,.
$$
To compute the signature on $[s,t]$, it suffices to start at the initial condition ${\X}_s = \mathbf{1}$.
\end{remark}

It follows automatically from the proof of Theorem \ref{fundthmSGRP} that, 
defining $\mathrm{MinExt}(N)$ as the class of sgrms which arise as minimal  extension of some $N$-sgrms over $\bR^d$, we have the inclusions
$$
       \mathrm{MinExt}(1) \subset \mathrm{MinExt}(2) \subset \cdots \subset \mathrm{MinExt}(N )  \subset \cdots \subset \{\text{sgrm over $\bR^d$} \}\,,
$$
\[\bigcup_{N\geq 1}\mathrm{MinExt}(N)\subset \{\text{sgrm over $\bR^d$} \}\,.\]
When $d>1$, all inclusions above are strict. For instance, take a non-finite Lie series, i.e. $\mathfrak{v} \in \mathcal{L}((\mathbb{R}^d)) \backslash \mathcal{L}(\mathbb{R}^d)$. Then $\X_{s,t} = \exp_{\otimes}(\mathfrak{v}(t-s))$ defines a sgrm, which is not a minimal extension of any $N$-sgrm. These strict inclusions motivate the following finer subset of sgrms.

\begin{definition}\label{defn_good_model_smooth}
A sgrm $\X$ over $\bR^d$ is called a {\bf good sgrm} if $\X= \mathrm{MinExt(\Y)}$ for some $N$-sgrm $\Y$,  $N \in \mathbb{N}$.
\end{definition}

Recalling  that a minimal extension has the same diagonal derivative as the underlying  $N$-sgrm, one  has the immediate  characterisation of good sgrms.
\begin{lemma} \label{lem:good} A sgrm $\X$ is good if and only if for all times $s\in [0,T]$ one has 
\begin{equation}\label{prop:diagonal_derivative_good}
\dot{\X}_{s,s} 
 \in \mathcal{L}(\bR^d)\,.
\end{equation}
\end{lemma} 

We link the notion of minimal extension with previous constructions in the literature.

\begin{example}\label{sec:example1}
 Every smooth path $Y: [0,T] \to \bR^d$ can be regarded (somewhat trivially) as $1$-sgrp $\Y = (1,Y)$. The solution to
$$
       \dot{\X} = \X \otimes {(\Y^{-1} \otimes_1 \dot{\Y})} = \X \otimes \dot{Y}, \qquad {\X}_0 = \mathbf{1} \in T((\bR^d))
 $$
is then precisely given by the stack of iterated integrals
$$
        \X_t = \mathrm{Sig}(Y|_{[0,t]}) = \left( 1 , \int dY, ... \,, \int (dY)^{\otimes n}, ...\right) \in T((\bR^d)) 
$$
with $n$-fold integration over the $n$-simplex over $[0,t]$, i.e. 
 $0 \le u_1 \le \dots \le u_n \le t$. The signature of $\Y$ on $[s,t]$ coincides then with the usual definition of signature of $Y$, see e.g.\ \cite{lyons2007differential}, modulo a choice of unit initial data at time $s$, which entails integration over simplices over $[s,t]$.  
\end{example}
 
\begin{example}\label{sec:example2}
Fix a step-$2$ Lie element over $\bR^d$,
$$
\mathfrak{v} = (\mathfrak{a},\mathfrak{b}) \in \bR^d \oplus [\bR^d, \bR^d] 
\equiv \cL^2 (\bR^d),
$$ 
where $[\bR^d, \bR^d]$ stands for the set of anti-symmetric $2$-tensors over $\bR^d$. Consider the example of a $2$-sgrm given by   
$$
    \Y_{s,t}= (1, \mathfrak{a}(t-s),\mathfrak{b}(t-s) + \tfrac{1}{2} \mathfrak{a}^{\otimes 2}(t-s)^2) = \exp_{\otimes_2} (\mathfrak{v}(t-s)) \in T^2 (\bR^d).
$$
One  easily sees that $\dot{\Y}_{t,t}= \mathfrak{v}$ for all $t\in [0,T]$. Since $\mathfrak{v}$ is constant in time, the differential equation 
$$
      \dot{\X} = \X \otimes \mathfrak{v}\,, \qquad {\X}_s = 1 \in T((\bR^d))
$$
has an explicit exponential solution at any time $t$ given by
\begin{equation} \label{equ:ConstSpeedRP}
            \X_{s,t} = \exp_\otimes ( \mathfrak{v}(t-s) ) = 1 + \mathfrak{v}(t-s) + \tfrac{1}{2}
            \mathfrak{v}^{\otimes 2}(t-s)^2 + \dots \in T((\bR^d))\, .
\end{equation}
This is precisely the signature of $\Y$ on $[s,t]$ and $\X$ is the minimal extension of $\Y$. Modulo a starting point in the group such a construction is well-known as a \textbf{log-linear rough path}, here the special case of second order logarithm. The case $\mathfrak{v} = (\mathfrak{a},0) $ is covered in Example \ref{sec:example1} via the constant velocity path $\dot{Y} \equiv\mathfrak{a}$. By taking $\mathfrak{v} = (0,\mathfrak{b})$ the minimal extension coincides with a \textbf{pure area rough path}, see \cite{galuppi2019rough} for an algebraic geometric perspective on these structures. The explicit exponential solution \eqref{equ:ConstSpeedRP} is possible here thanks to the constant velocity $\mathfrak{v}$. In a general situation, with time-dependent $\mathfrak{v}=\mathfrak{v} (t)$ the solution has exponential form given by Magnus expansion and   additional commutator terms will appear, see e.g. \cite{iserles1999solution}.
\end{example}

We finish this subsection by relating sgrp to (weakly) geometric rough paths in the sense of standard definitions as found e.g. in \cite{frizbook,hairer2015geometric}.
\begin{proposition}\label{prop:smooth_vs_normal}
One has the following properties:
\begin{itemize}
\item[(i)]Every $N$-sgrp $\X$ is a weakly $1/N$-Hölder weakly geometric rough path.  Consequently, $\X$ is also a $\gamma$-Hölder geometric rough path, for any $1/\gamma \in (N,N+1)$, and the set of $N$-sgrp is dense therein.
\item[(ii)] The minimal extension of some $N$-sgrp $\X$ to a  coincides with the Lyons lift of $\X$, as constructed e.g. in \cite[Ch. 9]{frizbook}.
\end{itemize}
\end{proposition}
\begin{proof}

(i) We only need to discuss the analytic regularity, which is formulated for the increments $\X_{s,t} = \X^{-1} \otimes \X_t$, i.e. the associated rough model.
Thanks to (b.iii), we have 
$|\langle\X_{s,t},w\rangle| \lesssim |t-s|  \lesssim |t-s|^{|w|(1/N)}$ uniformly over $s,t \in [0,T]$, using that $|w|\le N$. The consequence follows from well-known relations between geometric and weakly geometric rough paths \cite[Ch. 9]{frizbook}. The final density statement is clear, since already the minimal level-$N$ extensions of smooth paths (a.k.a. canonical lifts) are dense in this level-$N$ rough paths space, cf. \cite{frizbook}.

(ii) We remark that the minimal extension of $\X$ can be constructed as solving a linear differential equation driven by $\X$, in the sense of Definition \ref{def:DEsgrp} below, whereas the Lyons lift solves the analogous rough differential equations.
\end{proof}

\subsection{Canonical sum and minimal coupling of smooth geometric rough models}
\label{sec:srplus_geom}

As was pointed out, the state space of a sgrp is $G(\R^d)$, a non-linear subset of $T (( \R^d))$. In particular,  the pointwise sum $\X_t + \Y_t$ of any given   two sgrps $\X,\Y$ does not generally make sense  as a sgrp. From a classical rough path perspective, the obstruction to making sense of  the addition of rough path increments lies in its possible dependence on (a priori) missing mixed iterated integrals. For smooth rough paths, however, it turns out that there is a canonical way to add smooth geometric rough models and, more generally a scalar multiplication.
\begin{definition}\label{def:sum_and_scalar_multiplication} For any fixed sgrms $\X,\Y$ let $t\mapsto\Z_t \in T (( \bR^d ))$ be the Cartan development of $\dot\X_{s,s}+\dot\Y_{s,s}$, i.e. the unique solution to
$$
\dot\Z_t=  \Z_t\otimes \left(\dot\X_{t,t}+\dot\Y_{t,t}\right), \quad \Z_0 = \1.
$$
We then write $\Z := \X \srplus\Y$ for the associated sgrm and call it the \textbf{canonical sum} of $\X$ and $\Y$. For any $\lambda\in \mathbb{R}$ we define also the sgrm $\Z=\lambda\srscalar\X$ via the Cartan development of  $\lambda \dot\X_{s,s}$, we call it the \textbf{canonical scalar multiplication}.
\end{definition}
We will see later that this addition of sgrm overlaps non-trivially with the renormalization of rough path. Operations $\srscalar$ and $\srplus$ equip the space of sgrps with a vector space structure and by Lemma \ref{lem:good} good sgrps form a linear subspace.
For some general comments on these operations we refer to Remark \ref{rem:srscalar} in the more general framework of roughs path associated to a Hopf algebra. We simply remark that in the smooth geometric setting the canonical sum coincides with a  construction of Lyons for $p$-variation weakly geometric rough paths, see \cite[Section~3.3.1~B]{lyons1998}.
\begin{proposition}\label{prop:sum}
For any fixed couple of sgrms $\X,\Y$ a map $\Z\colon [0,T]^2\to G(\R^d)$ coincides with $\X \srplus \Y$ if and only if $\Z$ satisfies $\Z_{s,t}=\Z_{s,u}\otimes \Z_{u,t}$ for $s,u,t\in[0,T]$ and one has for any $s\in [0,T]$
 \begin{equation}\label{eq:prop:sum2}
\Z_{s,t}=\X_{s,t}\otimes \Y_{s,t}+R_{s,t}\,,
 \end{equation}
 for some $R_{s,t}\in T((\R^d))$ such that for all $x\in T(\R^d)$ one has  $\langle R_{s,t},x\rangle= o(|t-s|)$ as  $t\to s$. Moreover, we have the relations
 \begin{equation}\label{eq:prop:sum}
 \X_{s,t}\otimes\Y_{s,t}
 =\Y_{s,t}\otimes\X_{s,t}+r_{s,t}
 =\X_{s,t}+\Y_{s,t}-\mathbf{1}^*+r'_{s,t}\,,
 \end{equation}
 for some $r_{s,t}$, $r'_{s,t} \in  T((\R^d))$ such that for all $x\in T(\R^d)$ one has $\langle r_{s,t},x\rangle,\langle r'_{s,t},x\rangle= o(|t-s|)$ as $t\to s$.
\end{proposition} 
\begin{proof}
    See the more general Proposition \ref{prop:minimal_sum}.
\end{proof}

\begin{remark}\label{rem:srplus_honest_geom_rp}
Looking back at \cite[Section~3.3.1~B]{lyons1998} and Proposition \ref{prop:smooth_vs_normal}, we actually conclude that it is possible to sum a $N$-sgrm $\X$ to any general $\gamma$-Hölder weakly geometric rough path $\mathbf{W}$ for any $\gamma\in(0,1)$ in a canonical way,    though the construction is not obtained via diagonal derivatives but via the demand $(\X\srplus\mathbf{W})_{s,t}=\X_{s,t}\otimes\mathbf{W}_{s,t}+o(|t-s|)$ and a sewing lemma argument. This gives a strong motivation of looking at smooth rough paths as "universal perturbations" (accordingly to \cite[Section~3.3.1~B]{lyons1998}) of $\gamma$-Hölder weakly geometric rough paths. 
\end{remark}

The canonical sum can now be used to define a \textbf{minimal coupling} in the situation where we have a finite number of smooth geometric rough models $\X^i:[0,T]\to G(\mathbb{R}^{d_i})$. By putting $d=\sum_i d_i$ and fixing a canonical embedding $\mathbb{R}^{d_i}\subset \bR^d$ we can uniquely construct an injective $\otimes$ homomorphisms $\iota_i:T((\R^{d_i}))\to T((\R^d))$ that extends the embedding and sends $G(\bR^{d_i})$ into $ G(\bR^d)$. 
Then we can consider the sgrms $\iota_i\X^i:[0,T]\to G(\R^d)$ and we define the minimal coupling of the $\X^i$ as the canonical sum of the $\iota_i\X^i$
\begin{equation}\label{def:minimal_coupling}
(\X^1,\dots,\X^m)_{\text{min}}:=\iota_1\X^1\srplus\cdots\srplus \iota_m\X^m\,.
\end{equation}
The name minimal amounts to the fact that we are choosing a sgrm which involves the least possible information on mixed iterated integrals. This construction can also a partial generalization of Lyons-Young ``$(p,q)$'' pairing discussed in \cite{frizbook}.

\subsection{Differential equations driven by SGRP}
Let $\left( f_1, \dots , f_{d}\right) \in (\Vect^\infty(\bR^e))^{d}$ be a collection of smooth vector fields, with bounded derivatives of all orders, so that all stated operations and differential equations below make sense. For the empty word $\1$, set $f_{\1}=id$, the identity vector field. For a word $w=\ell_1...\ell_n$ with $|w| \ge 1$ letters $\ell_j \in \{1,...,d \}$, define the vector field
\begin{equation}\label{def_vector_field}
f_w :=f_{\ell_1}\vartriangleright( \ldots \vartriangleright(f_{\ell_{n-1}}\vartriangleright f_{\ell_n})\ldots)\,,
\end{equation}
where $\vartriangleright$ is the following operation of two vector fields: i.e. using Einstein summation over $i,j = 1,...,e$ we set
\begin{equation}\label{defn_pre_lie}
 f \vartriangleright g := f^i (\partial_i g^j) \partial_j \,,
\end{equation}
where $f=(f^i \partial_i),$ and $g = (g^j \partial_j)$. Using tensor calculus notation, we can equivalently set $f \vartriangleright g = (\nabla g) f$ where $\nabla g$ is the Jacobian matrix of $g$. The resulting family of maps $\{y\mapsto f_w(y)\}_{w}$ is represented with a  linear map $f\colon T ( \bR^d ) \to  \Vect^\infty(\bR^e)$ such that
\begin{equation}\label{map_f}
T ( \bR^d )\in  \mathbf{x} \mapsto  f_{\mathbf{x} } (\cdot) = \langle f (\cdot),\mathbf{x} \rangle,
\end{equation}
where the final pairing amounts to view $f$ as 
(formal) sum $\sum f_w e_w$, with summation over all words. Importantly, this map, restricted to Lie polynomials induces a Lie algebra morphism
$$
{{\mathcal L}}(\bR^d ) \ni \mathfrak{u} \mapsto f_\mathfrak{u} \in \Vect^\infty(\bR^e)\,,
$$
which is uniquely determined by his values on the canonical basis of $\mathbb{R}^d$, since $\mathcal{L}(\bR^d)$ is isomorphic to the free Lie algebra over $\mathbb{R}^d$. From an intrinsic geometric point of view, $f_\mathfrak{u}$ is  a vector field over $\mathbb{R}^e$ interpreted as a manifold, with its true for $f_\mathbf{x}$ 
only to the extend that one implicitly uses that flat connection on $\bR^e$.

In what follows, we equip $T(\bR^d)$ with an inner product structure, by declaring orthonormal the basis vectors $e_w \in T(\bR^d)$, induced by distinct words $w$. Viewing $T(\bR^d)$ as subspace of $T((\bR^d))$, this is consistent with the  {readily} used pairing $\langle \mathbf{x}, \mathbf{z} \rangle$ with $ \mathbf{x} \in T((\bR^d))$, $\mathbf{z} \in T(\bR^d)$. In particular $\{ e_w : |w| \le N \}$  yields an orthonormal   basis in the level-$N$ truncated space $T^N(\bR^d)$, with inner product $\langle \cdot, \cdot \rangle = \langle \cdot, \cdot \rangle_N$. We now introduce  a natural notion of differential equation associated to sgrms and a class of vector fields $\left( f_1, \dots ,f_{d}\right)$.
\begin{definition} \label{def:DEsgrp}
Let $\X$ be $N$-sgrm or a good sgrm such that $\dot{\X}_{s,s}\in\cL^N(\mathbb{R}^d)$. We say that a smooth path $Y\colon [0,T]\to \bR^e$  is the solution of a differential equation driven by $\X$ with vector fields $\left( f_1, \dots , f_{d}\right)\in (\Vect^\infty(\bR^e))^{d}$ if it satisfies the differential equation
\begin{equation}\label{eq:RDE}
\dot{Y}_s= \langle f (Y_s), \dot{\X}_{s,s}\rangle_N=\sum_{|w|\le N} f_w (Y_s)\langle \dot{\X}_{s,s},w\rangle\,,
\end{equation}
where $f:\mathbf{x}\mapsto f_{\mathbf{x}}$ is given in \eqref{def_vector_field}. We will refer to equation \eqref{eq:RDE}   with the shorthand notation
\begin{equation}\label{def:RDE}
dY=f(Y)d\mathbf{X}\,.
\end{equation}
\end{definition}
\begin{remark}
We remark that the equation \eqref{def:RDE} generalises the known notion of controlled ODE. Indeed, $\langle \dot{\X}_{s,s},\1\rangle=0$ for any $N$-sgrm or a good sgrm  and if $\X$ is the minimal $N$-extension of a smooth path $X=\sum_{i=1}^d X^i e_i$, obtained by classical iterated integration in Example \ref{sec:example1}, then Definition \ref{def:DEsgrp} collapses to the usual definition of controlled equation
$$
\dot{Y} = \sum_{i=1}^d f_i(Y) \dot{X}^i,
$$
thereby obtaining a consistent definition. However, the richer structure of smooth geometric rough paths in the examples \ref{sec:example2} extends this framework.
\end{remark}

Given a generic initial condition $Y_0\in \bR^e$ and vector fields $\left( f_1, \dots , f_{d}\right)$
on $\bR^e$, 
it follows from standard properties on classical differential equations that there exists a unique solution with initial condition $Y_0$. For a $N$-sgrp $\mathbf{X}$, the notion of solution of differential equation driven by a sgrp is furthermore consistent with the notion of rough differential equation solution à la Davie \cite{Davie08} when the driver $\mathbf{X}$ is interpreted as a $1/N$-Hölder weakly geometric rough path.
\begin{proposition}\label{prop:euler}
A path $Y\colon [0,T]\to \mathbb{R}^e$ is a solution of a differential equation driven by a $N$-sgrp  $\X$ with vector fields $\left( f_1, \dots , f_{d}\right)$ if and only if for all $s,t\in[0,T]$
\begin{equation}\label{eq:euler_RDE}
    Y_t-Y_s=\sum_{1\leq |w|\leq N}f_w(Y_s)\langle \X_{s,t},w\rangle+r_{s,t}\,
\end{equation}
where $\X_{s,t}$ is the geometric rough model associated to $\X_t$ and $r$ is a remainder such that $r_{s,t}= o(|t-s|)$ as $t\to s$.
\end{proposition}
\begin{proof}
 Supposing the property \eqref{eq:euler_RDE}, we first observe, since $f_{\mathbf{1}}=0$ and $\langle\X_{s,s},w\rangle=0$ for any non-empty word $w$,
 \begin{equation*}
     \lim_{t\downarrow s}\frac{Y_t-Y_s}{t-s}=\lim_{t\downarrow s}\sum_{1\leq |w|\leq N}f_w(Y_s)\big\langle \X_s^{-1}\otimes_N\frac{\X_t-\X_s}{t-s},w\big\rangle=\sum_{1\leq |w|\leq N}f_w(Y_s)\langle \dot\X_{s,s},w\rangle
 \end{equation*}
 We thus conclude that $Y$ is differentiable and obtain \eqref{eq:RDE}. 
 (With $f$ and $s \mapsto \dot{\X}_{s,s}$ smooth, it is then clear that $Y$ is not only differentiable but in fact smooth.)
Conversely, supposing that $Y$ is a solution of \eqref{eq:RDE}, we apply  Taylor's formula  to $Y$ and $t\to \langle\X_{s,t}, w\rangle$ for any $|w|\leq N$ obtaining for all $s,t\in [0,T]$.
\begin{equation*}
\begin{split}
    &Y_t-Y_s=\sum_{1\leq |w|\leq N}f_w(Y_s)\langle \dot{X}_{s,s},w\rangle(t-s)+r'_{s,t}=\sum_{1\leq |w|\leq N}f_w(Y_s)\langle X_{s,t},w\rangle+r''_{s,}
    \end{split}
\end{equation*}
for some $r'_{s,t}$, $r''_{s,t}$ satisfying both  $r'_{s,t}\,,r''_{s,t}= o(|t-s|)$ as $t\to s$.
\end{proof}

Using the properties of the diagonal derivative, we can restate the identity \eqref{eq:RDE} with respect to the Lie algebra ${{\mathcal L}}^N(\bR^d)$.

\begin{lemma}\label{lem:ONB}
Let $\X$ be $N$-sgrm or a good sgrm such that $\dot{\X}_{s,s}\in\cL^N(\mathbb{R}^d)$ and consider  $\mathfrak{B}^N$ an orthonormal basis for $\cL^N(\bR^d) \subset T^N(\bR^d)$. Then we have the equivalence
$$
dY = f (Y) d \X \qquad  \text{ iff }\qquad\dot{Y}_s  = \sum_{\mathfrak{u} \in  \mathfrak{B}^N} f_\mathfrak{u}  (Y_s)\langle  \dot{\X}_{s,s},\mathfrak{u} \rangle\,.
$$

\end{lemma} 
\begin{proof}
Complete $\mathfrak{B}=\mathfrak{B}^N$ to an orthonormal  basis $\bar{ \mathfrak{B}} = \mathfrak{B} \cup \mathfrak{B}^\perp$ of $T^N(\bR^d)$. In that basis
$$
\dot{Y}_s = \langle f (Y_s), \dot{\X}_{s,s}\rangle_N =\sum_{\mathbf{x} \in  \bar{ \mathfrak{B}}}   f_\mathbf{x}  (Y_s)\langle \dot{\X}_{s,s},\mathbf{x}    \rangle.
$$
We now observe that $\dot{\X}_{s,s} \in \cL^N(\bR^d)$, thanks to Proposition \ref{prop:lie1}. Consequently, there is no contribution from any $\mathbf{x} \in \mathfrak{B}^\perp \subset \cL^N(\bR^d)^\perp$.
\end{proof}
\begin{example}
In case $N=2$ a natural orthonormal basis for $\mathcal{L}^2(\mathbb{R}^d)$ is given by
\[\{ i, \; \frac{1}{\sqrt{2}}[i,j]\colon i,j \in \{1, \cdots,d \}, \; i<j\}\]
and the previous lemma asserts that $dY = f (Y) d \X$ is equivalent to the following higher order controlled equation
\[
\begin{split}
\dot{Y}_s &= \sum_{i=1}^d   f_i  (Y_s)\langle \dot{\X}_{s,s},i\rangle+  \sum_{i<j}  \frac{1}{2} f_{[i,j]} (Y_s)\langle \dot{\X}_{s,s},[i,j]\rangle\\&= \sum_{i=1}^d   f_i  (Y_s)\langle \dot{\X}_{s,s},i\rangle+  \sum_{i<j}  \frac{1}{2} [f_i, f_j]  (Y_s)\langle \dot{\X}_{s,s},[i,j]\rangle\,.
\end{split} 
\]
This also follows directly by applying Remark \ref{rk_N2} to the equation \eqref{eq:RDE} with $N=2$.
\end{example} 
\begin{remark} \label{rmk:RosaNoGo}
Even if a differential equation driven by a $N$-sgrm can be viewed as a Davie solution, we note however that Lemma \ref{lem:ONB} is a particular property of differential equations driven by $N$-smooth geometric rough paths which does not hold for RDEs driven by a generic  $1/N$-Hölder weakly geometric rough path.

Indeed, working with  one-dimensional $1/2$-Hölder weakly geometric rough path\footnote{Even though $d=1$ weakly geometric rough paths are an outlier in the sense that they are fully determined by the underlying path, they can be trivially embedded into rough paths spaces of higher dimension and thus such a counterexample also applies there.} and a generic vector field $f\in \Vect^\infty(\bR)$ one as the trivial identities
\begin{align*}
&\sum_{1\leq|w|\leq 2}f_w(Y_s)\langle\X_{s,t},w\rangle=f(Y_s)(X_t-X_s)+(f\cdot f')(Y_s)\frac{(X_t-X_s)^2}{2}\,,\\&
\sum_{\mathfrak{u} \in \mathfrak{B}^2}   f_\mathfrak{u}  (Y_s)\langle     \X_{s,t},\mathfrak{u}\rangle =f(Y_s)(X_t-X_s)\,.
\end{align*}
Therefore there is no $r_{s,t}= o(|t-s|)$ such that
\[\sum_{|w|\leq 2}f_w(Y_s)\langle\X_{s,t},w\rangle= \sum_{\mathfrak{u} \in \mathfrak{B}^2}   f_\mathfrak{u}  (Y_s)\langle     \X_{s,t},\mathfrak{u}\rangle+ r_{s,t}\,.\]
\end{remark}

\subsection{Algebraic renormalization of SGRP}\label{sec:alg_ren_geo}

 Consider a collection of Lie series $v = (v_1,\ldots, v_d)\subset \cL ((\bR^{d}))$, and let $T_v$ be the $\otimes$-endomorphism on $T((\bR^{d}))$, obtained by {extending the translation map}  $e_i \mapsto e_i + v_i$. We call $T_v$ the \textbf{translation map}. In case $v_i$ are all Lie polynomials, we remark we that  $T_v$ maps $T^{N}(\R^d)$ to $T^{M}(\R^d)$ with $M=N\cdot N'$ where $N'$ denotes the smallest integer such that $v_i\in\mathcal{L}^{N'}(\R^d)$. By restriction to Lie series, $\mathcal{L} ((\bR^d))\subset T((\bR^{d}))$, we can and will view   ${T}_v$ also as Lie algebra endomorphism, still denoted by $T_v$. We first give a dynamic view on the higher-order translation of sgrp.

\begin{theorem} \label{thm:DynRen}
\begin{enumerate}
\item[(i)] Let $v=(v_1,\dots,v_d)$ be a collection of elements in $\mathcal{L}((\R^d))$. Given a sgrp $\X$ over
$\bR^{d}$, 
the unique solution to 
\begin{equation}\label{eq:transl_ODE}
       \dot \Z_t = \Z_t \otimes  (T_{v} \dot{\X}_{t,t} )\, , \qquad \Z_0 = T_v \X_0 
\end{equation}
takes values in $G(\mathbb{R}^d)$ and is again a sgrp, and we have the explicit form
\begin{equation}\label{eq:expl_identities}
\Z_t =  {T}_v (\X_t)\,, \quad \Z_{s,t} =  {T}_v (\X_{s,t}).
\end{equation}
Starting with a sgrm $\X$ over $\R^d$, the above applies to $t\mapsto \X_{0,t}$ and we obtain a sgrm $\Z$  with explicit form given by
$$ \Z_{s,t} =  {T}_v (\X_{s,t}).$$

\item[(ii)] Let now $v=(v_1,\dots,v_d)$ be a collection of elements in $\mathcal{L}(\R^d)$ and $\X$ be a good sgrm over $\R^d$. Then $\Z$ as constructed in (i) is also a good sgrm. More specifically, let $N'$ be the smallest  integer such that $v_i\in\mathcal{L}^{N'}(\R^d)$ for all $i$ and assume $\X$ is the minimal extension of some $N$-sgrm $\Y$ over $\bR^{d}$. Let $M = N\cdot N'$. The unique solution to 
$$
       \dot \W_t = \W_t \otimes_{M}  (T_{v} \dot{\Y}_{t,t} ) , \qquad \W_0 = 1, 
$$
defines a $M$-sgrp, given by $\W_{s,t} = \W_s^{-1}\otimes_M \W_t$, which we call $\mathcal{T}_v[\Y]$.
Moreover, we have the explicit form
\begin{equation}\label{eq:not_local_transl}
\mathcal{T}_v[\Y]_{s,t}
=  {T}_v^{M} (\Y_{s,t}^{M})
=\proj_{M}{T}_v(\X_{s,t})
\end{equation}
with algebra endomorphism $T_v^M := \proj_{M} T_v\mathfrak{i}^M$ of $(T^M(\R^d),\otimes_M)$, using the (linear) embedding $\mathfrak{i}^M:\,T^M(\R^d)\to T((\R^d))$, and $\Y^M=\mathrm{MinExt}^{M}(\Y)$.
\end{enumerate}
\end{theorem}

\begin{remark}
Note that the renormalization map $\mathcal{T}_v$ for an $N$-sgrm is neither a linear nor a pointwise map, in contrast to $\X \to T_v\X$ for sgrps or sgrms.
\end{remark}
\begin{proof}
\emph{(i)} Fix a sgrp $\X$ and write 
$$
     \dot \X_t = \X_t \otimes (\X^{-1}_t\otimes \dot{\X}_t)= \X_t \otimes \dot{\X}_{t,t}. 
$$
Since $T_v$  commutes with derivation and is an algebra morphism,  then $\Z_t = T_v (\X_t)$  clearly satisfies the differential equation \eqref{eq:transl_ODE}. The algebraic properties of $T_v$ together with Proposition \ref{prop:lie1} imply  that $T_v(\X_t)$ belongs to $G(\mathbb{R}^d)$  and $T_v(\dot{\X}_{t,t})$ belongs to $ \mathcal{L}((\mathbb{R}^d))$ respectively. Therefore $T_v(\X_t)$ must coincide with the Cartan development of $T_v(\dot{\X}_{t,t})$, thereby  yielding  \eqref{eq:expl_identities}. Similar results hold with a sgrm. 
\medskip

\emph{(ii)} Fix a good sgrm and apply the part \emph{(i)} to the path $t\mapsto \X_{0,t} $. Since $v$ contains only Lie polynomials the map $T_v$ becomes an algebra morphism $T_v\colon T(\mathbb{R^d})\to T(\mathbb{R^d})$. Using  Lemma \ref{lem:good} and this last property one has $T_v(\dot{\X}_{t,t})\in \mathcal{L}(\mathbb{R}^d)$, which implies that $T_v(\X_{s,t})$ is a good sgrm. In the specific case of a collections of elements belonging to $\cL^{N'}(\mathbb{R}^d)$, we consider $\Y^M=\mathrm{MinExt}^{M}(\Y)$. The path $t\mapsto \Y^M_{0,t}=\Z_t $ will satisfy the following equation
\begin{equation}\label{eq:proof_transl}
\dot \Z_t = \Z_t \otimes_{M}  \dot{\Y}^M_{t,t}  , \qquad \Z_0 = \1, 
\end{equation}
Using stardard properties of tensors, $T_v^M$ is an algebra endomorphism of $(T^M(\R^d),\otimes_M)$. Indeed for any $x,y\in T^M(\R^d)$
\begin{align*}
    (T_v^M x)\otimes_M(T_v^M y)&=\proj_M((T_v^M x)\otimes (T_v^M y))=\proj_M((T_v x)\otimes (T_v y))\\&=\proj_M T_v(x\otimes y)
    =\proj_M T_v \mathfrak{i}^M \proj_M(x\otimes y)=T_v^M (x\otimes_M y),
\end{align*}
where we used in the fourth equality the fact that $T_v z\in J^{>M}$ for all $z\in J^{>M}$ ($T_v$ is not lowering the grade). Applying $T^M_v$ to both sides of equation \eqref{eq:proof_transl}, we obtain that $T^M_v(\Z_t)$ coincides with the Cartan development in $G^M(\mathbb{R}^d)$ of $T_v^M \dot{\Y}^M$.  By uniqueness of this  we can express $\mathcal{T}_v[\Y]_{s,t}$ as $T^M_v(\Z_s)^{-1}\otimes_MT^M_v(\Z_t)= {T}_v^{M} (\Y_{s,t}^{M})$, obtaining \eqref{eq:not_local_transl}. The identity ${T}_v^{M} (\Y_{s,t}^{M}) =\proj_{M}{T}_v(\X_{s,t})$ follows trivially.
\end{proof}

Combining the properties of translation operators $T_v$ with the differential equation \eqref{def:RDE}, we can  describe the effect  of translation on smooth geometric rough paths in the same way as \cite{Bruned2019}. Given $v=(v_1,\dots,v_d)$ a collection of elements in $\mathcal{L}(\R^d)$ and $(f_1\,, \cdots\,, f_d)\in (\Vect^\infty(\bR^e))^{d}$ we consider the collection of vector fields $ (f^v_1,\ldots, f^v_d)\in (\Vect^\infty(\bR^e))^{d}$ given for each $i=1, \cdots, d$ by 
\[
f^v_i= f_{i} + f_{v_i}\,,
\]
where $f_{v_i}$ was given in  \eqref{map_f}. Starting from this collection of translated vector fields on $d$ directions, we can consider the linear map $f^v\colon T ( \bR^d ) \to  \Vect^\infty(\bR^e)$, extending the family $(f^v_1,\ldots, f^v_d)$ on any word like in \eqref{def_vector_field}. This operation allows to transfer the action of the translation at the level of differential equations.
\begin{theorem}\label{thm:renormalisation_RDE}
Let  $\X$ be a good sgrm and $\W$ a $N$-sgrm. For any given collection $v = (v_1,\ldots, v_d)$ of elements in $\cL(\mathbb{R}^d)$ a path $Y\colon[0,T]\to\mathbb{R}^e $ solves one of the equation
$$
dY = f(Y) d (T_v(\X))\,,  \quad dY = f (Y) d (\mathcal{T}_v[\W])\,;
$$
if and only if it solves respectively
$$
dY = f^{v} (Y) d \X\,,  \quad dY = f^{v} (Y) d\W\,.
$$
\end{theorem}
\begin{remark}
The differential equations in the above statement are understood, in the sense of Definition \ref{def:DEsgrp}, as equations driven by a $N$-sgrm $\W$, a $(N\cdot N')$-sgrm $\mathcal{T}_v[\W]$ (when the directions $v = (v_1,\ldots, v_d)$ are all contained in $\cL^{N'}(\mathbb{R}^d)$) and  a good sgrm $T_v \X$, respectively.
\end{remark}
\begin{proof}
Writing a $\X$ as the minimal extension of a $M$-sgrm for some $M\in \mathbb{N}$, the result follows by showing only the equivalence for $\W$. We present here  two possible proofs.

(First argument)
Fix $s$ and $y= Y_s$ and $\mathfrak{w} := \dot{\W}_{s,s} \in \cL^N(\mathbb{R}^d)$ and  $M=N\cdot N'$, where $N'$ is the minimal integer such that $v_i\in\mathcal{L}^{N'}(\bR^{d})$ for all $i=1\,, \cdots\,, d$. By Theorem \ref{thm:DynRen}, the diagonal derivative of $\mathcal{T}_v[\W]$ at time $s$ is precisely $T_v \mathfrak{w} \in \cL^M(\bR^{d})$ up to minimal extension. Using Lemma \ref{lem:ONB}, the result follows once we check that
$$
\sum_{\mathfrak{u}\in   \mathfrak{B}^M} f_\mathfrak{u}  (y )\langle   T_v \mathfrak{w},\mathfrak{u}  \rangle = \sum_{\mathfrak{v} \in  \mathfrak{B}^N} f^v_\mathfrak{v}  (y)\langle   \mathfrak{w},\mathfrak{v}\rangle\,,
$$
To prove this identity, we use an argument already contained in \cite{Bruned2019}. Let  $f : \mathfrak{x} \mapsto f_{\mathfrak{x}}$ be the induced Lie algebra morphism from $\cL^M(\bR^{d})$ into $\Vect^\infty(\bR^e)$. Hence, whenever $f$ is smooth, the maps $f_{T_v}$ and $f^v$ are both Lie algebra morphisms from $\cL^M(\bR^{d})$ into $\Vect^{\infty}(\bR^e)$, which furthermore agree on the generators $e_i$. Thus $f_{T_v\mathfrak{x}}(z) = f^v_\mathfrak{x}(z)$ for any $z\in \mathbb{R}^e$ and any $\mathfrak{x}\in \cL^M(\bR^{d})$ thanks to the universal property of $\cL^M(\bR^{d})$  so
$$
\sum_{\mathfrak{u} \in \mathfrak{B}^M}f_{\mathfrak{u}} (y )\langle T_v \mathfrak{w},\mathfrak{u}  \rangle =    f_{T_v\mathfrak{w}}(y)   = f^v_{\mathfrak{w}}(y) =\sum_{\mathfrak{v}\in  \mathfrak{B}^N} f^v_\mathfrak{v} (y)\langle   \mathfrak{w},\mathfrak{v}\rangle\,.
$$

(Second argument) The identity $f_{T_v}(z)=f^{v}(z)$ does not hold simply over $\cL(\bR^{d})$ but over all $ T(\mathbb{R}^d)$. For that, we first show that in fact for any $u\in\mathcal{L}(\R^d)$ and any $x\in T^{>0}(\R^d)$ we have $f_u\vartriangleright f_x=f_{u\otimes x}$. We do so completely analogously to \cite[Lemma~2.5.11]{Preiss21},
proceeding by induction over the length of $u$. 
 For $u$ a letter, the statement holds by definition of $w\mapsto f_w$.
 Assume the statement holds for all $u$ of length $n$. 
 Then, to check it for all homogeneous Lie elements of word length $n+1$, 
 since left bracketings span the free Lie algebra,
 it suffices to look at Lie elements of the form $[u,i]$. 
 So,
 \begin{align*}
  f_{[u,i]\otimes x}
  &=f_{u\otimes i\otimes x}-f_{i\otimes u\otimes x}
  =f_u\vartriangleright(f_{i}\vartriangleright f_x)-f_i\vartriangleright(f_u\vartriangleright f_x)\\
  &=[f_u,f_i]_{\vartriangleright}\vartriangleright f_x
  =f_{[u,i]}\vartriangleright f_x,
 \end{align*}
 where for the third equality, we used the pre-Lie identity.
Then, since $v_i\in\mathcal{L}(\R^d)$, indeed we get, again via induction over the length of words,
\begin{equation*}
f_{iw}^v=f_i^v\vartriangleright f_{w}^v =f_{T_v i}\vartriangleright f_{T_v w} 
=f_{v_i}\vartriangleright f_{T_v w}=f_{v_i\otimes T_v(w)}  =f_{T_v (iw)}\,.
\end{equation*}
We finally obtain the general identity by linearity. Thus for any $\mathbf{x} \in T^N(\mathbb{R}^d)$ we have
\begin{equation*}
\sum_{w}\langle \mathbf{x},w\rangle f_w^v 
=f_\mathbf{x}^v=f_{T_v \mathbf{x}}= \sum_{w}\langle T_v \mathbf{x},w\rangle f_w\,.
\end{equation*}
Then, by Proposition \ref{prop:euler} we have the desired equivalence of differential equations through the following equality, for any smooth path $Y:[0,T]\to\mathbb{R}^e$ and $s,t\in [0,T]$, using the notation $\mathbf{W}^\infty:=\mathrm{MinExt}(\mathbf{W})$ one has
\begin{align*}
&Y_t-Y_s=\sum_{1\leq |w| \leq N} \scal{\mathbf{W}_{s,t},w} f^v_{w}(Y_s)
 =\sum_{1\leq |w|} \scal{\proj_{ N }\mathbf{W}^\infty_{s,t},w} f^v_{w}(Y_s)\\&=\sum_{1\leq |w|} \scal{T_v\proj_{N}\mathbf{W}^\infty_{s,t},w} f_{w}(Y_s)=\sum_{1\leq |w|} \scal{\proj_{ M 
 }T_v\mathbf{W}^\infty_{s,t},w}f_{w}(Y_s)+R_{s,t}\\&=\sum_{1\leq |w|\leq M 
 }\scal{(\mathcal{T}_v[\mathbf{W}])_{s,t},w} f_{w}(Y_s)+R_{s,t},
\end{align*}
with
\begin{align*}
 R_{s,t}
 &=\sum_{1\leq|w|\leq M 
 }\scal{T_v\proj_{ N 
 }\mathbf{W}^\infty_{s,t}-T_v\mathbf{W}^\infty_{s,t},w}f_{w}(Y_s)
 \\&=-\sum_{1\leq |w|\leq M}
 \scal{\proj_{>N 
 }\mathbf{W}^\infty_{s,t},T_v^*w}f_{w}(Y_s)=o(|t-s|)\,,
\end{align*}
as $\langle\mathbf{W}^\infty_{s,t},w\rangle=O(|t-s|^2)$ for all $w$ with $|w|>N$, since then $(s,t)\mapsto\langle\mathbf{W}^\infty_{s,t},w\rangle$ is a smooth function on the compact domain $[0,T]^2$ with $\langle\mathbf{W}^\infty_{t,t},w\rangle=0$ and $\partial_s|_{s=t}\langle\mathbf{W}^\infty_{s,t},w\rangle=\langle\dot{\mathbf{W}}^\infty_{t,t},w\rangle=0$.
\end{proof}
\begin{remark}
The second argument in the above proof is valid for genuine (non-smooth) $\gamma$-H\"older weakly geometric rough paths, with $N = \floor{1/\gamma}$, $Y:[0,T]\to\mathbb{R}^e$ an arbitrary continuous path and choosing $\mathrm{MinExt}(\mathbf{W})$ as the Lyons extension of $\mathbf{W}$, see Proposition \ref{prop:smooth_vs_normal}. There, it again correctly identifies the respective rough differential equations via Davies' expansion. 
This argument first presented in \cite[Section~2.5.2]{Preiss21} fills a gap left in the proof of \cite[Theorem~38]{Bruned2019} where it is misleadingly suggested that a Lie algebraic argument can be used.
Indeed, as noted in Remark \ref{rmk:RosaNoGo} there is no Davie-type definition of RDE solution where the sum is restricted to a Lie basis. 
\end{remark}
As final application of the section we combine Theorems \ref{thm:DynRen} and \ref{thm:renormalisation_RDE} with the notion of minimal coupling and sum to obtain a ``smooth geometric'' renormalisation along a time component, see  \cite[Thm. 38]{Bruned2019}.  Consider a time-space path of the form \[
\bar{X}_t = (t, X_t) = (X^0_t,X^1_t) \in \R \oplus \R^d = \R^{1+d}
\]  
for some smooth path $X^1\colon [0,T]\to \R^d$ (the first component of  canonical basis of $\R^{d+1}$ will be denoted by $e_0$). Assume also that there exists a good sgrm  $\X^1:[0,T]^2 \to G (\R^d)$ extending $X^1$. Then we can introduce the  usual signature of time component $\X^0:[0,T]^2 \to G (\R)$  defined by
$$
\X^0_{s,t}=\exp_{\otimes}((t-s)e_0)
$$ 
and the choice 
$$\bar{\X}:=(\X^0,\X^1)_{\text{min}}$$
constitutes a canonical way to extend $\bar{X}$. Choosing then a Lie polynomial $v_0\in \cL(\R^{d+1})$, we can  also consider the translation map $T_{v_0}$ obtained by determined by the identity action on $e_1,\dots, e_d$, and $e_0 \mapsto e_0 + v_0$. Then we can explicitly describe $T_{v_0}\bar{\X}$ and differential equations driven by that.
\begin{corollary}\label{cor_time_transl}
Let  $\X^1$ be a good sgrm and $v_0\in\mathcal{L}(A)$. Then the translation $t\mapsto \hat{T}_{v_0}\bar{\X}_{0,t}$ can be described as the sum $\mathbf{V}_0\srplus \bar{\X}$, where $\mathbf{V}_0$ is given by
\[(\mathbf{V}_0)_{s,t}=\exp_{\otimes}(v_0(t-s))\,.\]
Moreover, for any given family $\left( f_{\hat{0}}, f_a \colon a\in A \right)$ one has the equivalence 
$$
dY = \bar{f} (Y) d (\hat{T}_{u_{0}}(\bar{\X}))\qquad \text{iff} \qquad dY = (f_0 + \bar{f}_{\Phi^*_Hu_0})(Y) dt+ f (Y) d \X^1\,,
$$
where $\bar{f}\colon T(\R^{d+1})\to \Vect^\infty(\bR^e)$ and $f\colon T(\R^{d+1})\to \Vect^\infty(\bR^e)$ are defined like in \eqref{map_f} starting respectively from $( f_0, \dots , f_{d})$ and $( f_1, \dots , f_{d})$.
\end{corollary}
\begin{proof}
The proof follows by applying the results in the section. By construction of $\bar{\X}$ one has the following diagonal derivative $\dot{\bar{\X}}_{t,t}=e_0+\dot\X^1_{t,t}$. Then by construction of $T_{v_0}$ and Theorem \ref{thm:DynRen}, we have
\begin{equation}\label{eq:srplusv0_geom}
\begin{split}
\frac{d}{dt}(T_{v_0}(\bar{\X})_{s,t})|_{t=s}=T_{v_0}(e_0+\dot\X^1_{s,s})=v_0+e_0+\dot\X^1_{s,s}=\frac{d}{dt}(\mathbf{V}_0\srplus\bar{\X})_{s,t}|_{t=s}\,.
\end{split}
\end{equation}
Since the diagonal derivative identifies uniquely $T_{v_0}(\bar{\X})$ we conclude. Passing to the differential equation, we remark from Definition \ref{def:DEsgrp} that a smooth path $Y\colon [0,T]\to \mathbb{R}^e$ is a  solution of $dY=\bar{f}(Y)d\bar{\X}$ if and only if
\[dY = f_0(Y) dt+ f (Y) d \X^1\,.\]
Moreover, the map $\bar{f}^{v_0}$ satisfies trivially $\bar{f}^{v_0}_0=f_0 + \bar f_{v_0}$, $\bar{f}^{v_0}_i=f_i$ for any $i=1, \cdots,d $. We conclude then by Theorem \ref{thm:renormalisation_RDE}.
\end{proof}

\section{Smooth quasi-geometric rough paths}\label{sec:sqgrp}

In the sequel, we extend the previous results and notions of smooth rough paths in the context of quasi-geometric rough paths, the natural generalisation of  rough paths when we consider an underlying quasi-shuffle algebra, see \cite{hoffman2000}. Quasi-geometric rough paths were first introduced in talks by David Kelly and properly studied  in \cite{Bel2020a}.

\subsection{Quasi-shuffle algebras}\label{sec:quasishuffle}
We briefly review quasi-shuffle algebras in the version of \cite{hoffman2000}. Their origin  can be traced back to \cite{Cartier} and  we also refer the reader to \cite{FoissyPatras,MP}.  We shall use Hoffman’s isomorphism in the context of stochastic  integration (’It\^o vs. Stratonovich’) which was discussed in detail in \cite{kurusch15}. 
 
Let us start from a generic vector space $\bR^A$  where  $A$ is a locally finite (with respect to the weight function $\omega$ below) set called \textbf{alphabet}. To link this space with the previous section, we   assume the inclusion $\{1\,\cdots, d\}\subset A $,  which induces a canonical inclusion $\bR^d\subset \bR^A$. Given an alphabet $A$, with a slight abuse of notation we will denote by $T(A)$ and $T((A))$  the spaces of tensor polynomials and tensor series over $\bR^A$ respectively. To simplify the notation, in what follows we will also identify elements of $T((A))$ as words, as explained in the previous section i.e.
\[e_{a_1}\otimes  \cdots \otimes e_{a_n}\longleftrightarrow a_1\cdots a_n\,,\]
for any $a_1\,, \cdots \,, a_n\in A$. Moreover, we use the same symbol $\otimes$ to denote the tensor product operation of $T((A))$ and the ``external" algebraic tensor product of two vector spaces.

 The quasi-shuffle product relies on the existence of a \textbf{commutative bracket}, i.e. a map $\{\,,\}\colon \overline{A}\times \overline{A}\to \overline{A}$, where $\overline{A}=A\cup \{0\} $ such that
\begin{equation}\label{equ:bracket1}
\{ a ,0\}= \{ 0 ,a\}=0\,, \quad \{ a ,b\}= \{ b ,a\}\,, \quad \{ a ,\{b, c\}\}=  \{ \{a ,b\}, c\} .
\end{equation}
We adopt the notation $\{ a \} =a $, for $a \in A$, and more generally, for any  word $w= a_1\cdots a_n$ we set  $\{a_1\cdots a_n\}:=\{a_1 ,\{\cdots \{a_{n-1}, a_{n}\}\}\cdots\}$ independently of the order of the letters in the word and the brackets. Together with the choice of a commutative bracket $\{\,,\}$ we  also fix a weight  $\omega$ which is compatible with $\{\,,\}$, i.e. $\omega$ is a function $\omega\colon A\to  \bN^*$ satisfying
 \begin{equation}\label{equ:bracket2}
\omega(\{ab\})=\omega(a)+ \omega(b)
\end{equation}
when $\{ab\}\neq 0$. Note that   there could be several weights  compatible with a given bracket and the definition
of bracket does not need a weight.

Given a word $w=a_1  \dots a_n$  we denote by $|w|=n$  the  length of $w$ and  introduce its weighted length $\norm{w}= \omega(a_1)+ \cdots + \omega(a_n)$. Since $\omega(a)\geq 1$ for all $a\in A$ we have trivially $ |w|\leq \norm{w}$. The weighted length induces a grading on $T(A)$. To stress this property, we also use the notations $T_{\omega}(A)$ and $T_{\omega}((A))$ when we grade these vector spaces according to the weighted length. e.g.
\[
T_{\omega}(A)= \bigoplus_{n=0}^{\infty} \langle\text{\{words $w\colon \Vert w \Vert =n$\}}\rangle\,.
\]
Moreover, for any  $N\in \mathbb{N}$ we define the truncated space $T_\omega^N(A)$ by taking words  of weighted length at most $N$, we denote the projection on $T_\omega^N(A)$ by $\proj_{N,\omega}$. Using the weighted length, we can  view $T^N_{\omega}(A)$ as a quotient algebra and introduce the corresponding truncated tensor product $\otimes_{N, \omega}$. The pairing $\langle\,, \rangle$ between  $T((\bR^d))$ and $T(\bR^d)$ generalises  to  the pairing $\langle\,, \rangle \colon T_{\omega}((A))\times T_{\omega}(A)\to \bR$    for any alphabet $A$.  Its restriction    $T(A)\times T(A) \subset T((A))\times T(A)$ yields a scalar product on $T(A)$. We denote by  $\langle \cdot, \cdot \rangle_N$  the restriction to truncated spaces $T^N_{\omega}(A)\times T^N_{\omega}(A)$. As before, the set  $\{ w : \|w\| \le N \}$ is an orthonormal basis of $T^N_{\omega}(A)$.  Combining $\{\,,\}$ and the shuffle product, we introduce the quasi-shuffle product. 
\begin{definition}
The \textbf{quasi-shuffle product} $\qshuffle$  is  the unique bi-linear map in $T(A)$ satisfying the relation $\1 \,\qshuffle\, v= v\, \qshuffle\,\1 =v$ for any $v\in T(A) $ and the recursive identity
\begin{equation}\label{quasi_shuffle_recursive}
v\,a \qshuffle w\,b= (v \qshuffle w\,b)\, a + (v\,a \qshuffle w)\, b+  (v\qshuffle w)\{a b\} \,,
\end{equation}
for any  words $v,w\in T(A)$ and letters $a,b\in A$.
\end{definition}
\begin{example}
\label{rmk:shuffle}
If $A= \{1\,, \cdots\,, d\}$ and the commutative bracket is trivial, i.e. $\{a  b\}=0$ for any $a,b\in A$, then   the recursive relation \eqref{quasi_shuffle_recursive} reduces to 
\[va \qshuffle wb= (v \qshuffle wb)a + (va \qshuffle w)b\,.\]
so that $\qshuffle$ coincides with the standard shuffle product $\shuffle$ in \eqref{shuffle_recursive}. Observe that setting $\omega(i)=1$ for all $i \in A$, one has $\norm{w}=|w|$ and  $T_{\omega}(\{1\cdots d\})$ is exactly $T(\bR^d)$ but in general we can consider the shuffle algebra over a generic alphabet with a weight which is not identically $1$, see e.g. in \cite{hairer2015geometric}.
\end{example}

\begin{example} \label{exe:bracket}

 For fixed integers $M,d \ge 1$,   we define the alphabet $\bA_{M}^{d}$ as
\[\bA_{M}^{d}=\left\{\ga\in  \bN^d\setminus \{0\}\colon \sum_{i=1}^d \ga_i\leq M\right\}\,,\]
 as well as the weight $ \omega({\ga})= \sum_{i=1}^d \ga_i$. We also define the bracket $\{,\}_{M}\colon \overline{\bA_{M}^{d}}\times\overline{\bA_{M}^{d}}\to\overline{\bA_{M}^{d}}$ as the unique bilinear map satisfying  the property
\begin{equation}\label{eq:def_[]N}
\{\ga,\gb\}_{M}: =\left\{
	\begin{array}{ll}
    \ga +\gb  & \mbox{if } \ga+\gb\in \bA_{M}^{d}\,, \\
	0 & \mbox{otherwise},
	\end{array}
	\right.
\end{equation}
for any  $\alpha, \beta\in \bA_{M}^{d}$. The sum arising in this expression  is the intrinsic sum between multi-indices. The set  $\{1,\cdots, d\} $ embeds in $\bA_{M}^{d}$ for any $M\geq 0$ by simply sending every $i \in\{1\,, \cdots\,, d\}$ to the $i$-th element of the canonical basis. Using this inclusion  the letter $\{112\}$ is identified with the multi-index $(2,1,0,...0)$ and e have $\{ i_1\cdots i_n\}=0$ for any  $n>M $ and $i_j\in \{1,\cdots,d\}$. It follows from standard properties on multi-indices that $\{, \}$ and $\omega$ satisfy properties \eqref{equ:bracket1} and \eqref{equ:bracket2}. The associated quasi-shuffle algebra  $T_{\omega}(\bA_{M}^{d})$ is used to define the notion of quasi-geometric bracket extension in \cite{Bel2020a}. We can also drop the finite index $M$ and consider the infinite alphabet $\bA^{d}= \bN^d\setminus \{0\}$ with the same weight. This alphabet is then isomorphic to the free semi-group generated over $d$ elements. The associated quasi-shuffle algebra was intensively used in \cite{Tapia20}.
\end{example}

For any choice of $\{\,,\}$ the quasi-shuffle product $\qshuffle$ is a well-defined commutative product on $T(A)$, see  \cite{hoffman2000}. It further defines  a product on the  graded algebra $T_{\omega}(A)$ for  a given compatible weight  function  $\omega\colon A\to  \bN^*$. The product $\shuffle$ is also  compatible with the grading on $T_{\omega}(A)$, giving rise to two  algebra structures on the same graded vector space.

From a Hopf algebraic point of view, see Section \ref{Hopf_algebra_section} for the general framework, we can equip with $T_{\omega}(A)$ with the deconcatenation coproduct $\Delta\colon T_{\omega}(A) \to  T_{\omega}(A)  \otimes T_{\omega}(A) $ and the counit  $\1^*\colon T_{\omega}(A) \to \bR$ defined respectively  by
\[\Delta (a_{i_1}\cdots a_{i_n})=\1\otimes a_{i_1}\cdots a_{i_n}+ \sum_{k=1}^n a_{i_1}\cdots a_{i_k}\otimes a_{i_{k+1}}\cdots a_{i_n}+ a_{i_1}\cdots a_{i_n} \otimes \1,\]
\[  \1^*(w): =\left\{
	\begin{array}{ll}
    1 & \mbox{if } w=\1\,, \\
	0 & \mbox{otherwise}
	\end{array}
	\right.\]
Both the triples $(T_{\omega}(A),\qshuffle, \Delta)$ and  $(T_{\omega}(A),\shuffle, \Delta)$  are  graded Hopf algebras, see \cite{hoffman2000}. A fundamental result is the existence of an explicit isomorphism between these two Hopf algebras for any commutative bracket $\{\,,\}$, see \cite[Thm.~3.3]{hoffman2000}.

\begin{theorem}[Hoffman Isomorphism]\label{thm:hoffmann_iso} 
For any commutative bracket $\{\,,\} $ we define the Hoffman ``exponential'' and  ``logarithm'' $\Phi_H$, $\Psi_H\colon  T(A) \to  T(A)$ on any word as
\begin{equation}\label{defn_exp}
\Phi_H(w)=\sum_{I\in C(\abs{w})}\frac{1}{I!}\{w\}_I\,, \quad\Psi_H(w)=\sum_{I\in C(\abs{w})}\frac{(-1)^{\abs{w}-|I|}}{I}\{w\}_I\,,
\end{equation}
 where $C(\abs{w})$ is the set of compositions of order  $\abs{w}$ i.e. the multi-indices $I=(i_1,\cdots, i_m)$ such that  $ i_1+ \cdots+ i_m= \abs{w}$ and 
\[I!= i_1!\cdots i_m!\,,\quad I= i_1\cdots i_m\,, \quad  |I|= m\,.\]
Moreover, we denote by $\{w\}_I$  the word
\[ \{w\}_I:= \{w_1\cdots w_{i_1}\}\{w_{i_1+1}\cdots w_{i_1+i_2}\}\cdots\{w_{i_1+\cdots +i_{m-1}+1} \cdots w_n\}\,.\]
The map $\Phi_H\colon (T_{\omega}(A), \shuffle, \Delta )\to (T_\omega(A), \qshuffle, \Delta)$ is an isomorphism of graded Hopf algebras such that $\Psi_H= \Phi_H^{-1}$.
\end{theorem}
With $\Phi_H$ and $\Psi_H$ as in \eqref{defn_exp},   two  explicit formulae were shown in \cite{hoffman2000} for the adjoint maps $ \Phi_H^*, \Psi_H^* \colon T_{\omega}(A)\to T_{\omega}(A)$ with respect to scalar product $\langle \,, \rangle$. Indeed, for any word $w=a_1\cdots a_n$ we have
\begin{equation}\label{defn_exp_dual1}
 \Phi^*_H(w)= \Phi^*_H (a_1) \cdots  \Phi^*_H(a_n)\,, \quad \Psi^*_H(w)= \Psi^*_H (a_1) \cdots \Psi^*_H(a_n)\,
\end{equation}
and for each letter $a\in A$,   the following identity holds:
\begin{align}\label{defn_exp_dual_2}
\Phi^*_H(a)=\sum_{n\geq 1}\sum_{\{a_1\cdots a_n\}=a}\frac{1}{n!}a_1\cdots a_n\,,  \quad\Psi^*_H(a)=\sum_{n\geq 1}\sum_{\{a_1\cdots a_n\}=a}\frac{(-1)^{n-1}}{n}a_1\cdots a_n\,,
\end{align}
where the underlying set of summation $ \{a_1\cdots a_n\}=a$ involves all possible ways of writing the letter $a$ in terms of brackets of the word $a_1\cdots a_n$.  By simply considering  the graded dual of the shuffle and quasi-shuffle Hopf algebras (see section \ref{Hopf_algebra_section}) in Theorem \ref{thm:hoffmann_iso}, the map $\Phi_H^*$ becomes an isomorphism   $\Phi_H^*\colon (T_\omega(A), \otimes, \Delta_{\qshuffle})\to (T_\omega(A), \otimes, \Delta_{\shuffle})$ where   $\Delta_{\qshuffle},\Delta_{\shuffle}\colon T_{\omega}(A)\to T_{\omega}(A)\otimes T_{\omega}(A)$ are the dual coproducts associated to the shuffle and quasi-shuffle operation. To extend $\Phi^*_H$ and $\Psi_H^*$ to the whole set $T((A))$, we observe that  the concatenation product $\otimes$ easily extends to $T_{\omega}((A))$. However, the coproducts  $\Delta_{\qshuffle}$ and $\Delta_{\shuffle}$ do not make sense on $T_{\omega}((A))\otimes T_{\omega}((A))$, see e.g. \cite[Thm.3.6]{Preiss16} for a counterexample on $T((\R^1))$. They need to be extended to   the completed tensor space, see \cite[Sec.~1.4]{reutenauer1993free}
\[T_{\omega}((A)){\overline{\otimes}}T_{\omega}((A)):= \prod_{k,m=0}^{\infty} \langle \text{\{words $w\colon \Vert w \Vert =k$\}}\rangle  \otimes \langle \text{\{words $w \colon \Vert w \Vert =m$}\rangle\,, \] leading to two  topological graded coalgebras $(T_\omega((A)),\Delta_{\qshuffle})$ and $(T_\omega((A)),\Delta_{\shuffle})$.

\begin{proposition}\label{prop:dual}
The maps $\Phi^*_H$ and $\Psi^*_H$ defined by \eqref{defn_exp_dual1} and \eqref{defn_exp_dual_2}   uniquely extend  to two linear maps $\Phi^*_H,\Psi^*_H \colon T_{\omega}((A))\to T_{\omega}((A))$, for which we shall use the same notation.  The map $\Phi^*_H$ gives rise to  a graded automorphism of  the algebra $(T_\omega((A)), \otimes)$  as well as an isomorphism between the two topological graded coalgebras $(T_\omega((A)),\Delta_{\qshuffle}) $ and $(T_\omega((A)), \Delta_{\shuffle})$. In both cases one has $\Psi_H^*= (\Phi_H^*)^{-1}$.
\end{proposition}

\begin{proof}
It is sufficient to   build the extension of $\Phi^*_H$. Combining properties \eqref{defn_exp_dual1} and \eqref{defn_exp_dual_2}, for any series $S=\sum_{w} \alpha_w w\in T_{\omega}((A))$ we define
\[\Phi^*_H(S):=\sum_{w}\sum_{n_1\geq 1}\sum_{\{a_1^1\cdots a_{n_1}^1\}=w_1} \cdots\sum_{n_m\geq 1}\sum_{\{a_1^m\cdots a_{n_m}^m\}=w_m}\alpha_w \frac{1}{n_1!}\cdots\frac{1}{n_m!} a_1^1\cdots a_{n_m}^m\,,\]
where $w=w_1\cdots w_m$. Thanks to property  \eqref{equ:bracket2} and the assumption $\omega(A)\subset \bN^*$, the terms in the  series $\Phi^*_H(S)$ are locally finite, i.e. for any word $u$ there is only a finite number of non-zero terms in the bracket $\langle \Phi^*_H(S),u\rangle$. It follows   from property \eqref{defn_exp_dual1} that $\Phi^*_H$ extended to  infinite series is also an algebra morphism.  The other properties easily follow   from the fact that $T_{\omega}(A)$ and $T_{\omega}(A)\otimes T_{\omega}(A)$ are dense subsets of $T_{\omega}((A))$ and $T_{\omega}((A)){\overline{\otimes}}T_{\omega}((A)).$ 
\end{proof}

The  Hoffman isomorphism allows to define Lie groups and Lie algebras from the corresponding Lie groups and Lie algebras built from the  corresponding shuffle structures. 

In what follows, we use the notations $\mathcal{L}^N(A)$, $G^N(A)$, $\mathcal{L}\left( A\right) $,  $\mathcal{L}\left((A)\right)$ and $G(A)= \exp_{\otimes} \cL((A))$ to denote respectively  the $N$-step Lie polynomials, the free step-$N$ Lie group, the Lie polynomials and Lie series generated by $\bR^A$ and $G(A)= \exp_{\otimes} \cL((A))$, where $\exp_{\otimes}$ is given in \eqref{exp_tensor}. In case we fix a generic weight $\omega$  and $N \in \mathbb{N}$, we denote by $\mathcal{L}^{N}_{\omega}\left(A\right)$ and $G^N_{\omega}(A)$ the following quotients
\[\mathcal{L}^{N}_{\omega}\left(A\right)= \mathcal{L}\left(A\right)/\mathcal{L}^{\omega}_{>N}\left(A\right)\,,\quad G^{N}_{\omega}\left(A\right)= G\left(A\right)/G^{\omega}_{>N}\left(A\right)\]
where  $\mathcal{L}^{\omega}_{>N}$ and $ G^{\omega}_{>N}\left(A\right)$ are respectively the Lie ideal and normal subgroup
\[\mathcal{L}^{\omega}_{>N}\left(A\right):=\langle\{ w : \|w\| > N \}\rangle \cap \mathcal{L}\left(A\right)\,, \quad G^{\omega}_{>N}:=\langle\{ w : \|w\| > N \}\rangle \cap G\left(A\right)\,.\]
(These last two properties follow trivially from \eqref{equ:bracket2}). Moreover, by passing to the quotient one has the identification
\begin{equation}\label{eq:recallGN}
\begin{split}
&G^{N}_{\omega}\left(A\right)=\\& \{ \mathbf{x} \in T_{\omega}^N(A): \langle \mathbf{x},v \shuffle w\rangle = \langle \mathbf{x},v\rangle\langle\mathbf{x},w\rangle\,
\text{on words $w,v$ s.t. $ \Vert w\Vert+ \Vert v\Vert\leq N $} \}
\end{split}
\end{equation}
and the property $G^{N}_{\omega}(A) =\exp_{\otimes_{N,\omega}}(\mathcal{L}^{N}_{\omega}(A))$ where $\exp_{\otimes_{N,\omega}}\colon T^N_{\omega}(A)\to T^N_{\omega}(A)$ is the exponential associated with the truncated product $\otimes_{N, \omega}$
\begin{equation}\label{defn_truncated_exp_weighted}
\exp_{\otimes_{N,\omega}}v= \sum_{n\geq 0}^{N}\frac{v^{\otimes_{N, \omega} ^n}}{n!}\,.
\end{equation}
thereby obtaining that $\mathcal{L}^{N}_{\omega}(A)$ is the Lie algebra of $G^{N}_{\omega}(A)$. Similar objects can be defined in the quasi-shuffle context via the map $ \Psi^*_H$. The proof of the following corollary is a straightforward application of Proposition \ref{prop:dual} combined with the usual properties of Lie series and the Group $G(A)$.
\begin{corollary}\label{prop:hoffmann_iso_lie}
\begin{enumerate}
\item[(i)] The set $\hat{G}(A) :=  \Psi_H^*\,  G(A)=\{ \Psi_H^*\,\beta, \beta \in G(A)\}$ with the operation $\otimes$ is a group which has the following description
\begin{equation}\label{}
\hat{G}(A)= \{ \mathbf{x} \in T((A)): \langle \mathbf{x},v \qshuffle w\rangle = \langle \mathbf{x},v\rangle\langle\mathbf{x},w\rangle
\text{ for all words $w,v$} \}\,.
\end{equation}
We call $\hat{G}(A)$ the group of \textbf{quasi-shuffle characters}.
\item[(ii)] The sets $\hat{\cL}((A))= \Psi_H^*\cL((A))$ and $\hat{\cL}(A):=  \Psi_H^*\,  {\cL}(A)$ with the commutator of $\otimes$ are two Lie algebras, which satisfy $\hat{\cL}(A)\subset  \hat{\cL}((A))$ and $\hat{G}(A)=\exp_{\otimes}\mathcal{L}\left((A)\right)$. We call them the \textbf{quasi-shuffle Lie series} and \textbf{quasi-shuffle Lie polynomials} respectively and they coincide with the free Lie algebra and Lie series generated by the set $\mathfrak{A}=\{\Psi_H^*(a)\colon a\in A\}$.
\item[(iii)] The set  $\hat{G}^N_{\omega}(A):= \Psi_H^*G^N_{\omega}(A)$ with the operation $\otimes_{N, \omega}$ is a Lie group whose Lie algebra is given by $\hat{\mathcal{L}}^{N}_{\omega}\left(A\right):= \Psi_H^*\mathcal{L}^{N}_{\omega}\left(A\right)$ and one has $\hat{G}^N_{\omega}(A)= \exp_{\otimes_{N,\omega}}(\hat{\cL}^{N}_{\omega}(A))$ and the identification
\begin{equation}\label{eq:recallG^N}
\begin{split}
&\hat{G}^{N}_{\omega}\left(A\right)=\\& \{ \mathbf{x} \in T_{\omega}^N(A): \langle \mathbf{x},v \qshuffle w\rangle = \langle \mathbf{x},v\rangle\langle\mathbf{x},w\rangle\,
\text{on words $w,v$ s.t. $ \Vert w\Vert+ \Vert v\Vert\leq N $} \}.
\end{split}
\end{equation}
\end{enumerate}
\end{corollary}

\begin{example}
Let us spell out the primitive elements of  $\hat{\cL}_{\omega}^3(\bA_{2}^{d})$. Starting from the definition of $\cL_{\omega}^3(\bA_{2}^{d})$ and using the notations $[\,,]$ for the Lie bracket, one has that $\cL_{\omega}^3(\bA_{2}^{d})$ is generated as vector space by the following elements 
\[\bigg\{i,  [i,j], \{ij\}, [[i,j],k],[i,\{jk\}]\,\bigg\}\,,\]
where $i,j,k \in \{1, \cdots , d\}$. It follows from the definition of $\Psi^*_H$ that
\[\Psi_H^*(i)= i\,, \quad \Psi_H^*(\{ij\})=\{ij\} - \frac{1}{2}\left( ij+ij\right)\mathbf{1}_{i\neq j}-  \frac{1}{2} ii\,\mathbf{1}_{i= j}\,.\]
We therefore obtain that  $\hat{\cL}_{\omega}^3(\bA_{2}^{d})$ is generated as vector space by
\[\bigg\{i,  [i,j], \Psi_H^*(\{ij\}), [[i,j],k],[i,\Psi_H^*(\{jk\})]\,\bigg\}\,.\]
\end{example}

\subsection{Smooth  quasi-geometric rough paths}\label{subsec:sqgrp}
The notion and properties of smooth rough paths  can be transposed  to the quasi-shuffle context leading to equivalent concepts.
\begin{definition}\label{def:smooth-quasi-geom}
We call level-$N$ {\bf smooth quasi-geometric rough path} over the alphabet $A$   with weight $\omega$ (in short: $N$-{\bf sqgrp})  any non zero path $\X: [0,T] \to T^N_{\omega}(A)$ satisfying the following properties:
\begin{itemize}
\item[(a' i)] for all times $t \in [0,T]$ and all words $v$ and $w$ such that $\norm{w}+\norm{v}\le N$
\begin{equation}\label{eq:qshufflerel}
\langle\X_{t},v \qshuffle w\rangle = \langle \X_{t},v\rangle\langle\X_{t},w\rangle\,.
\end{equation}
\item[(a' ii)] For every word of weighted length $\norm{w} \le N$, the map $t \mapsto \langle\X_{t},w\rangle$ is smooth and we write $\dot{\X}_t $ for the derivative of $\X$.
\end{itemize}
We call level-$N$ {\bf smooth quasi-geometric rough model} (in short: $N$-{\bf sqgrm}) any map $\X: [0,T]^2 \to T^N_{\omega}(A)$ such that the property \eqref{eq:qshufflerel} holds for all $\X_{s,t}$ and the properties (b.ii), (b.iii) in Definition \ref{defn_model_smooth} hold on a set of weighted words with the operation $\otimes_{N, \omega}$. By {\bf smooth quasi-geometric rough path (model)} (in short: {\bf sqgrp} and {\bf sqgrm}) we mean a path (map) with values in the full space of tensor series $T_{\omega}((A))$ where \eqref{eq:qshufflerel} hold for any $w,v\in T_{\omega}((A))$ and relation (b.ii) hold with $\otimes$.
\end{definition}

It follows immediately from Definition \ref{def:smooth-quasi-geom} that  sqgrps and $N$-sqgrps are identified with smooth paths with values in the groups $\hat{G}(A)$ or $\hat{G}^N_{\omega}(A)$ and we can pass to the associated models by considering the increments with respect to $\otimes_{N, \omega}$. Similarly we can easily adapt the notion of extension of a sqgrp and the diagonal derivative of a sqgrm, which takes value in $\hat{\mathcal{L}}((A))$  or $\hat{\mathcal{L}}^N_{\omega}(A)$. Since the shuffle product is a specific case of the quasi-shuffle product, Definition \ref{def:smooth-quasi-geom}  extends also the notions of $N$-sgrps and  $N$-sgrms  in a presence of a generic alphabet $A$ and weight $\omega$, we will call these  objects \textbf{weighted $N$-sgrps} and \textbf{weighted $N$-sgrms}. 

Thanks to the properties of the Hoffman isomorphism, we can easily relate smooth (resp. truncated) quasi-geometric rough paths and models with some corresponding weighted geometric objects by simply composing these objects with  the  appropriate version of the Hoffman isomorphism.

\begin{theorem}\label{thm:fund_quasi_geo}
 Let $N\geq 0$ and $\X\colon[0,T]\to T^N_{\omega}(A)$. Then $\X$ is a  $N$-sqgrp if and only if $\hat{\X}=\Phi_H^*\X$ is a weighted $N$-sgrp. Conversely, $\hat{\X}$ is a weighted $N$-sgrp if and only if $\X=\Psi_H^*\hat{\X}$ is a  $N$-sqgrp. The same result applies to $N$-sqgrms, sqgrps and sqgrms.
\end{theorem}

\begin{proof}
The result follows easily   from  combining the properties of the maps $\Phi_H^*$ and $\Psi_H^*$ in Theorem \ref{thm:hoffmann_iso} with the conditions in Definition \ref{def:smooth-geom} and Definition \ref{def:smooth-quasi-geom}. E.g. starting from a $N$-sqgrp $\X$ one has 
\[
\langle\hat{\X}_{t},v \shuffle w\rangle =\langle\X_{t},\Phi_H(v \shuffle w)\rangle=\langle\X_{t},\Phi_H(v) \qshuffle \Phi_H(w) \rangle =\]
\[=\langle \X_{t},\Phi_H(v)\rangle\langle\X_{t},\Phi_H(w)\rangle=\langle \hat{\X}_{t},v\rangle\langle\hat{\X}_{t},w\rangle \]
for all words with joint weighted length $\norm{w}+\norm{v}\le N$. Moreover,  restricting  $\Phi^*_H\X_t$ to all words with joint weighted length $\norm{w}\le N$, it will be a finite linear combination of the functions $\langle\X_t, v\rangle$ with $\norm{v}\leq N$, which are all smooth functions. The same considerations apply  to $N$-sqgrms, sqgrps and sqgrms trivially. The converse case with $\Psi_H^*$ follows from the properties in Corollary \ref{prop:hoffmann_iso_lie}.
\end{proof}

Thanks to this one-to-one correspondence between quasi-geometric rough paths and geometric rough paths, we can import all the constructions of the previous section by simply checking that  they preserve the maps $\Phi_H^*$ and $\Psi_H^*$. In particular,   the notion of minimal extension for sgrms  can easily be transposed to sqgrms.

\begin{proposition}
Given an $N$-sqgrm  $\Y$ for some $N \in \bN$, there exists exactly one sqgrm extension $\X$  of $\Y$  which is minimal in the sense that for all $s\in[0,T]$ one has
$$
\dot{\X}_{s,s} \in \hat{\cL}^N_{\omega} (A)\,.
$$
We call $q\mathrm{MinExt(\Y)} := \X$ the {\bf quasi-geometric minimal extension} of $\Y$ and also $q\mathrm{MinExt^{N'}(Y)} := \proj_{N',\omega} \X$, for $N' > N$, the $N'$-minimal extension of $\Y$. 
For a fixed interval $[s,t] \subset [0,T]$ it holds that $\X_{s,t}$ only depends on $\{ \Y|_{[u,v]}: s\le u \le v \le t \}$ and we introduce the {\bf quasi-signature} of $\Y$ on $[s,t]$ by $q\mathrm{Sig}(\Y|_{[s,t]}) := \X_{s,t} \in \hat{G}(A)\,$. Moreover, we have the identities
\begin{equation}\label{eq:qsig_sig}
q\mathrm{MinExt(\Y)}= \Psi^*_H \mathrm{MinExt(\Phi^*_H\Y)}\,,\quad  q\mathrm{Sig}(\Y|_{[s,t]})=\Psi^*_H\mathrm{Sig}(\Phi^*_H\Y|_{[s,t]})\,.
\end{equation}
A sqgrm $\X$ is called a {\bf good sqgrm} if it satisfies $\X= q\mathrm{MinExt(\Y)}$ for some $N$-sqgrm $\Y$.

\end{proposition}
\begin{remark}
The existence and uniqueness of a minimal extension for sqgrg hold  on general grounds thanks to  properties of Hopf algebras, see Theorem \ref{fundthmSGRPH}. However, Hoffman isomorphism gives us in \eqref{eq:qsig_sig} an explicit solution of the minimal extension in terms of the minimal extension related to a geometric rough path.
\end{remark}
\begin{proof}
Thanks to Theorem \ref{thm:fund_quasi_geo}, the map $\hat{\Y}=\Phi^*_H\Y$  is a weighted $N$-sgrp. Since the proof of Theorem \ref{fundthmSGRP} applies to any alphabet $A$ and any weight on letters $\omega$, there exists a unique minimal extension $\hat{\X}$ of $\hat{\Y}$. Beyond the fact that the function $\X:=\Psi_H^*\hat{\X}$ is  a quasi-geometric rough path, its diagonal derivative satisfies
\[\dot{\X}_{s,s} = \Psi_H^*\dot{\hat{\X}}_{s,s} =\Psi_H^*\dot{\hat{\Y}}_{s,s}=\Psi_H^*\Phi_H^*\dot{\Y}_{s,s} = \dot{\Y}_{s,s}\,.\]
The uniqueness of $\hat{\X}$ follows from the uniqueness of the minimal extension in the geometric case and the isomorphism properties of $\Psi^*_H$ and $\Phi^*_H$. From this uniqueness we deduce also the identities \eqref{eq:qsig_sig} straightforwardly.
\end{proof}

In the same way as for geometric rough paths, the notion of smooth quasi-geometric rough path is consistent with its  equivalent $\gamma$-H\"older version as introduced in \cite{Bel2020a}. The proof of the following Corolllary is left to the reader.
\begin{corollary}
Every $N$-sqgrm $\X$ is a $1/N$-Hölder quasi-geometric rough path and its minimal extension $\X$ coincides with the lift of $\X$, as constructed in  \cite[Prop.~3.9]{Bel2020a}.
\end{corollary}
\begin{remark}\label{rk-weakly}
Even if in the literature there is no difference between $\gamma$-H\"older quasi-geometric rough paths and weakly $\gamma$-H\"older quasi-geometric rough paths, we can easily adapt the techniques in \cite{frizbook}  to introduce properly these concepts and to prove that $N$-sqgrm $\hat{\X}$ are weakly $1/N$-H\"older quasi-geometric rough paths and they are dense in the set of $\gamma$-H\"older quasi-geometric rough paths for  $1/\gamma \in (N,N+1)$. 
\end{remark}

\label{sec:srplus_quasi}
The existence of a minimal extension allows two define a vector space structure over sqgrms and weighted sgrps.  Indeed given two  smooth quasi-geometric rough models $\X,\Y:\,[0,T]^2\to\hat{G}(A)$ and $\lambda\in \mathbb{R}$, we define their canonical sum $\X\srplus\Y$  and their scalar multiplication $\lambda\srscalar\X$ like in Definition \ref{def:sum_and_scalar_multiplication}.
Since the operation of taking the diagonal derivative of a sqgrm commutes with the linear maps $\Phi_H^*$ and $\Psi_H^*$, from Proposition \ref{prop:minimal_sum} we deduce the following corollary.
\begin{corollary}\label{cor:srplus_quasi}
 Let $\X,\Y$ be smooth quasi-geometric rough paths/models. Then the map  $(s,t)\to \X_{s,t}\otimes\Y_{s,t}$ and $\X\srplus\Y$ satisfy the same relations in \eqref{eq:prop:sum} and  \eqref{eq:prop:sum2} respectively. Moreover $\Phi_H^*$ is a linear isomorphism with inverse $\Psi_H^*$ between the vector space of sqgrms and weighted sgrms endowed with $\srplus$ and $\srscalar$.
\end{corollary} 
\begin{remark}\label{min_quasi}
The sum of smooth quasi-geometric rough paths/models can be used to define a notion of minimal coupling like in \eqref{def:minimal_coupling}. In this case, the elements of the sum are smooth geometric rough paths/models $\X^i:[0,T]\to\hat{G}(A_i)$,  where  $A_i,    i=1, \cdots, m$   form  a partition of $A$ where each  $A_i$ is  closed under the commutative bracket $\{\,,\}$. 
\end{remark}

\subsection{Differential equations driven by SQGRP}\label{sec:alg_ren_qgeo} 
 For renormalisation purposes in the context of differential equations, we introduce the notion of differential equation driven by a smooth quasi-geometric rough path. Since we replaced our initial directions on $\mathbb{R}^d$ with an alphabet $A$, we assume the existence of a wider collection of smooth vector fields $\left( f_a\colon a\in A \right)\subset \Vect^\infty(\bR^e)$. Moreover, 
 in the same way as  $\{1, \cdots, d\}\subset A$ we   further fix a family   $\left( f_1, \cdots, f_d\right)\subset\left( f_a\colon a\in A \right)$. As before, we can apply  the operation $\vartriangleright$ to the elements $\left( f_a\colon a\in A \right)$ obtaining again a linear  map $f\colon T_{\omega}(A)\to \Vect^\infty(\bR^e)$ defined by \eqref{def_vector_field}. This map induces a  morphism of Lie algebras $f\colon \cL(A)\to\Vect^\infty(\bR^e)$, which is uniquely determined by his values on $A$. By means of the adjoint Hoffman exponential $\Phi^*_H$, we can uniquely define a map similar to $f$, which preserves the quasi-shuffle Lie polynomials.
\begin{proposition}
Given a family of vector fields $\left( f_a\colon a\in A \right)\subset\Vect^\infty(\bR^e)$, there exists a map $\hat{f}\colon T_{\omega}(A)\to \Vect^\infty(\bR^e)$ whose restriction on $\hat{\cL}(A)$ is the unique morphism of Lie algebras $\hat{f}\colon \hat{\cL}(A)\to \Vect^\infty(\bR^e)$ which satisfies $\hat{f}(\Psi^*_H(a))=f_a$ for any $a\in A$. This map is explicitly given by $\hat{f}= f_{\Phi^*_H}$.
\end{proposition}
\begin{proof}
By definition of $\hat{\cL}(A)$, The map $f_{\Phi^*_H}$ satisfies clearly the properties of the statement. Moreover, thanks to the free structure of $\hat{\cL}(A)$ described in Corollary \ref{prop:hoffmann_iso_lie}, the  condition  $\hat{f}(\Psi^*_H(a))=f_a$   uniquely determines a morphism of Lie algebras $\hat{f}\colon \hat{\cL}(A)\to \Vect^\infty(\bR^e)$. 
\end{proof}
Using the explicit definitions of $f$ and $\Phi^*_H$, we can then describe $\hat{f}$ as the unique linear map satisfying the conditions 
\begin{equation}\label{vec_field_2}
\hat{f}_w=(\hat{f}_{a_1}\vartriangleright( \ldots \vartriangleright(\hat{f}_{a_{n-1}}\vartriangleright \hat{f}_{a_n})\ldots)\,, \quad \hat{f}_a=\sum_{n\geq 1}\sum_{\{a_1\cdots a_n\}=a}\frac{1}{n!}f_{a_1\cdots a_n}\,,
\end{equation}
for any words $w=a_1\cdots a_n$ and $a\in A$. We are now ready to introduce the notion of differential equation in the quasi-geometric context.
\begin{definition}\label{def:QDE}
Let $\X$ be $N$-sqgrm or a good sgrm such that $\dot{\X}_{s,s}\in\cL^N_{\omega}(A)$.  We say that a smooth path $Y\colon [0,T]\to \bR^e$ is the solution of a differential equation driven by $\X$  and the vector fields $\left( f_a\colon a\in A \right)$ if it satisfies 
\begin{equation}\label{eq:qRDE2}
   \dot{Y}_s = \langle \hat{f} (Y_s), \dot{\X}_{s,s}\rangle_N =\sum_{\norm{w}\le N} 
   \hat{f}_w (Y_s)\langle  
   \dot{\X}_{s,s},w 
   \rangle\,.
\end{equation}
where $\hat{f}:\mathbf{x}\mapsto \hat{f}_{\mathbf{x}}$ is given in \eqref{vec_field_2}. We will refer to equation \eqref{eq:qRDE2}   as
\begin{equation}\label{eq:qRDE1}
dY=\hat{f}(Y)d\mathbf{X}\,.
\end{equation} 
\end{definition}
\begin{remark}
This definition leaves the open problem for future work how to canonically define a notion of differential equation, when only given  $d$ vector fields $f_1,\dots,f_d$, without the need for an arbitrary extension to $(f_a: a\in A)$, at best in a way that it is consistent with the geometric and branched $d$ dimensional setting.
\end{remark}

Combining the results in Proposition \ref{prop:euler}, Lemma \ref{lem:ONB}  and  Proposition \ref{prop:dual}, we can summarise three equivalent characterisations of \eqref{eq:qRDE1} in the quasi-geometric context.
\begin{proposition}\label{lem:quasi_ONB}
Given a path  $Y\colon [0,T]\to \bR^e$ a $N$-sqgrp $\mathbf{X}$ and a family of vector fields $\left( f_a\colon a\in A \right)$, one has the following equivalent conditions:
\begin{enumerate}[i)]
\item $Y$ solves $dY=\hat{f}(Y)d\mathbf{X}$.
\item $Y$ solves $dY = f (Y) d \hat{\X}$,  where $\hat{\X}=\Phi_H^* \X$ is the $N$-level weighted geometric rough path obtained via Theorem \ref{thm:fund_quasi_geo}.
\item For all $s,t\in[0,T]$
\begin{equation}\label{eq:euler_RDE2}
Y_t-Y_s=\sum_{1\leq \Vert w\Vert\leq N}\hat{f}_w(Y_s)\langle \X_{s,t},w\rangle+r_{s,t}\,,
\end{equation}
where $\X$ is the smooth quasi-geometric rough model associated to $\X$ and $r$ is a remainder such that $r_{s,t}= o(|t-s|)$ as $t\to s$.
\item For any  $\hat{\mathfrak{B}}^N$ orthonormal basis of $\hat{\cL}^N_{\omega}(A) \subset T^N_{\omega}(A)$ one has
$$
\dot{Y}_t  = \sum_{\mathfrak{u} \in  \hat{\mathfrak{B}}^N} \hat{f}_{\mathfrak{u}}  (Y_s)\langle     \dot{\X}_{t,t},\mathfrak{u}\rangle\,.
$$
\end{enumerate}
The results $i)$, $ii)$ and $iii)$ hold also if $\X$ is a good sqgrm.
\end{proposition}
\begin{proof}
The equivalence $i)\Leftrightarrow iii)$ follows in the same way as in Proposition \ref{prop:euler} and the equivalence $i)\Leftrightarrow iv)$ is a consequence of the property $\dot{\X}_{s,s}\in \hat{\cL}_{\omega}^N(A)$. To prove the last equivalence $i)\Leftrightarrow ii)$  it is sufficient to prove the identity
\begin{equation}\label{eq:Lie_Ito_strato}
\sum_{\mathfrak{u} \in  \hat{\mathfrak{B}}^N} \hat{f}_{\mathfrak{u}}  (y)\langle     \mathfrak{x},\mathfrak{u}\rangle= \sum_{\mathfrak{v} \in  \mathfrak{B}^N} f_{\mathfrak{v}}  (y)\langle     \Phi^*_H\mathfrak{x},\mathfrak{v}\rangle
\end{equation}
for any  $\mathfrak{x}\in \hat{\mathcal{L}}_{\omega}^N(A)$, $y\in \mathbb{R}^e$ and any orthonormal basis $\mathfrak{B}^N$ of $\mathcal{L}_{\omega}^N(A)$. However, by definition of $\hat{f}$ and using the property $\Phi^*_H\mathfrak{x}\in  \mathcal{L}_{\omega}^N(A)$ one has 
\[
\sum_{\mathfrak{u} \in  \hat{\mathfrak{B}}^N} \hat{f}_{\mathfrak{u}}  (y)\langle     \mathfrak{x},\mathfrak{u}\rangle=\hat{f}_{\mathfrak{x}} (y)=  f_{\Phi^*_H(\mathfrak{x})}(y) =\sum_{\mathfrak{v} \in  \mathfrak{B}^N} f_{\mathfrak{v}}  (y)\langle     \Phi^*_H\mathfrak{x},\mathfrak{v}\rangle\,.
\]
The case wit a smooth model follows straightforwardly.
\end{proof} 

\begin{example}\label{Ito_strato}
Let us check how some properties of Proposition \ref{lem:quasi_ONB} appear when we fix $A=\bA_2^d$ from example \ref{exe:bracket}, a sqgrm  $\X$ over  $A$ and a generic family of vector fields $\left( f_1, \cdots, f_d\right)$ extended to $A$ by setting $f_{\{ij\}}\equiv 0$. The differential equation  
\[dY= f(Y_t)d\hat{\X}\]
is equivalent to
\begin{equation}\label{last_check} \dot{Y}_s=  \sum_{|w|\le N} f_w (Y_s)\langle \dot{\hat{\X}}_{s,s},w   \rangle\,,
\end{equation}
where $w$ is a word with values in the alphabet $\{1, \cdots , d\}$. Using  the identity $\hat{\X}:= \Phi^*_H \X$, we can rewrite \eqref{last_check} as
\[ \dot{Y}_s=  \sum_{|w|\le N} f_w (Y_s)\langle \dot{\X}_{s,s},\Phi_Hw   \rangle.\]
It follows from the definition of $\Phi_H$ in \eqref{defn_exp} that it is possible to write down a general combinatorial formula for $\Phi_H(w)$, see \cite[Prop.~4.10]{kurusch15}
\[\Phi(w)=  \sum_{u\in \{w\}}\frac{1}{2^{|w|-|u|}} u\,, \]
where $\{w\}$ consists of the words we can construct from $w$ by successively replacing any neighbouring pairs $ij$ in $w$ by $\{ij\}$. Switching the sum on $w$  and $u$, equation \eqref{last_check}  becomes
\begin{equation}\label{last_check_2}
\dot{Y}_s=   \sum_{|w|\le N}\sum_{u\in \{w\}}\frac{1}{2^{|w|-|u|}}  f_w (Y_s)\langle \dot{\hat{\X}}_{s,s},u  \rangle=\sum_{\|u\|\le N}\sum_{\{w\}_p= u}\frac{1}{2^{|w|-|u|}}  f_w (Y_s)\langle \dot{\hat{\X}}_{s,s},u \rangle\,,
\end{equation}
where the  sum over $\{w\}_p= u $ involves all the ways one can write a word $u$ as $\{w\}_I$ of a word $w$ whose letters are in the alphabet $\{1, \cdots  , d\}$ and $I \in C(|w|)$. The element
\[\sum_{\{w\}_p= u}\frac{1}{2^{|w|-|u|}}  f_w\,, \]
then corresponds to the expression $\hat{f}_u$ for any word $u$. 

Using the shorthand notations $ \langle \dot{\X}_{s,s},a   \rangle=\langle \dot{\hat{\X}}_{s,s},a   \rangle=\dot{X}_s^a$ for any $a\in A$, the expressions \eqref{last_check} and \eqref{last_check_2} in case $N=2$ become respectively
\[\dot{Y}_s=\sum_{i=1}^d f_i (Y_s) \dot{X}_s^i   +\sum_{i,j=1}^d f_{ij} (Y_s)\langle \dot{\hat{\X}}_{s,s},ij   \rangle \] 
and
\[\begin{split}
&\dot{Y}_s=\sum_{i=1}^d \hat{f}_i (Y_s)\dot{X}_s^i   + \sum_{i,j=1}^d \hat{f}_{ij} (Y_s)\langle\dot{\X}_{s,s},ij\rangle+\sum_{i\leq j}^d \hat{f}_{\{ij\}} (Y_s)\dot{X}_s^{\{ij\}} \\& =\sum_{i=1}^d f_i (Y_s)\dot{X}_s^i+ \sum_{i,j=1}^d f_{ij} (Y_s)\langle \dot{\X}_{s,s},ij   \rangle + \frac{1}{2}\left(\sum_{i=1}^d f_{ii}(Y_s)\dot{X}_s^{\{ii\}}+\sum_{i<j}^d (f_{ij}+f_{ji})(Y_s)\dot{X}_s^{\{ij\}} \right)\\& =\sum_{i=1}^d f_i (Y_s)\dot{X}_s^i+ \sum_{i,j=1}^d f_{ij} (Y_s)\langle \dot{\X}_{s,s},ij   \rangle+  \frac{1}{2}\sum_{i,j=1}^d f_{ij} (Y_s)\dot{X}_s^{\{ij\}}\,.
\end{split} \] 
Interpreting $X^i$ and  $X^{\{ij\}}$ as the components of a semimartingale and his quadratic variations, we obtain an algebraic version of the   It\^o-Stratonovich correction among semi-martingales at level of SDEs, see  \cite{kurusch15} for further applications.
\end{example}
\subsection{Algebraic renormalisation of SQGRP}
We now want to adapt the notion of translation discussed for  geometric rough paths to the context of   quasi-geometric rough paths.  Adapting the same arguments in the section \ref{sec:alg_ren_geo}, for any family $(v_a \colon a\in A )\subset \cL ((A))$ we can easily define a translation map  $T_v\colon T_{\omega}((A))\to T_{\omega}((A))$  defined on any alphabet $A$. In the case of a family $(v_a \colon a\in A )\subset \cL (A) $ it follows from the properties of any weight $\omega$ that $T_v$ sends  $T^{N}_{\omega}(A)$ to $T^{M}_{\omega}(A)$ with $M=N\cdot N'$, where $N'$ is the smallest integer such that $v_i\in\mathcal{L}^{N'}_{\omega}(A)$. Passing to the quasi- shuffle context, the natural direction to perform a  translation must  be done along the Lie algebra $\hat{\cL} ((A))$. To define  a proper notion, Corollary  \ref{prop:hoffmann_iso_lie} tells us the set  $\hat{\cL} (A)$  consists of Lie series obtained from the set $\mathfrak{A}= \{\Psi^*_H(a)\colon a\in A\}$. Therefore,  fixing a subset $u = (u_a \colon a\in A )\subset \hat{\cL} ((A))$, there is a unique  morphism of Lie algebras $\hat{T}_{u}\colon \hat{\cL} ((A))\to \hat{\cL} ((A))$ such that
\[\hat{T}_{u}\left(\Psi_H^*(a)\right)=  \Psi_H^*(a) + u_a\,.\]
 Using the translation maps $T_v$ and the maps $\Psi^*_H$, $\Phi^*_H$ we  see that $\hat{T}_{u}$ uniquely extends   to an endomorphism of $T_{\omega}((A))$ with respect to the product $\otimes$.
 
\begin{theorem}\label{thm:transl_quasi}
For any collection $u = (u_a \colon a\in A )\subset \hat{\cL} ((A))$ there exists a unique $\otimes$-endomorphism over $T_{\omega}((A))$ which extends $\hat{T}_{u}\colon  \hat{\cL} ((A))\to \hat{\cL} ((A))$. We call it the  \textbf{quasi-translation map} and we denote it by the same symbol $\hat{T}_{u}$. In particular, one has the explicit relation 
\begin{equation}\label{eq:transl_quasi}
    \hat{T}_{u}= \Psi^*_H T_{\hat{u}}\Phi^*_H\,,
\end{equation}  
where $\hat{u}$ is given by the family  $\hat{u}= (\Phi^*_H(u_a) \colon a\in A )\subset\cL ((A))$. Moreover, if $(u_a \colon a\in A )\subset\hat{\cL} (A) $ $\hat{T}_u$ sends  $T^{N}_{\omega}(A)$ to $T^{M}_{\omega}(A)$ with $M=N\cdot N'$, where $N'$ is the smallest integer such that $v_i\in\hat{\mathcal{L}}^{N'}_{\omega}(A)$.
\end{theorem}
\begin{proof}
Thanks to Proposition \ref{prop:dual} and the properties of the translation map, the map $\Psi^*_H T_{\hat{u}}\Phi^*_H$ is a $\otimes$-endomorphism satisfying
\[\Psi^*_H T_{\hat{u}}\Phi^*_H(\Psi_H^*(a))= \Psi^*_HT_{\hat{u}}(a)=\Psi^*_H(a+\Phi_H^*(u_a))=\Psi^*_H(a)+u_a\,.\]
Moreover, from  Corollary \ref{prop:hoffmann_iso_lie} we also deduce that $\Psi^*_H  T_{\hat{u}}\Phi^*_H$ is a Lie endomorphism of $\hat{\cL}((A))$. Therefore we conclude from the freeness of $\hat{\cL}((A))$ that the definition \eqref{eq:transl_quasi} is a bona fide extension of $\hat{T}_u$. Concerning the uniqueness of $\hat{T}_u$, let $\Gamma\colon T_{\omega}((A))\to T_{\omega}((A))$ be an $\otimes$-endomorphism extending $\hat{T}_u$. Considering the map $\Phi^*_H \Gamma\Psi^*_H$ we can easily check on any $a\in A$ the property
\[ \Phi^*_H \Gamma \Psi^*_H(a)= a+ \Phi^*_H (u_a)\,.\]
Using the fundamental property of $T((A))$, $\Phi^*_H \Gamma \Psi^*_H $ must coincide with $T_{\hat{u}}$ and we obtain the uniqueness. The last properties of $\hat{T}_u$ follows for the fact that $\Phi^*_H$ and $\Psi^*_H$ are graded maps.
\end{proof}
\begin{remark}\label{rem:open_prob_quasi_translation}
This procedure still leaves open the  question as how to define   an endomorphism  $ T_u: T_{\omega}((A))\to T_{\omega}((A))$   starting from a mere family $u=(u_1,\dots,u_d)$ such that $T_u i=i+u_i$ for $i=1,\dots,d$.  This further question is   motivated by the fact that such translations coming from  (only) $d$ choices of Lie elements exist in a canonical way in the geometric and branched case, see also \cite{Bruned2019}. It would therefore be desirable to have such a translation map which is consistent with both the geometric and the branched $d$-dimensional setting.
\end{remark}

The quasi-translation maps now at hand can be applied to sqgrps/sqgrms and their truncated versions as in Theorem \ref{thm:DynRen}. Indeed for any $u = (u_a \colon a\in A )\subset \hat{\cL} ((A))$ and sqgrp/sqgrm $\X$ we can actually define the compositions $\hat{T}_u (\X_t)$ and  $\hat{T}_u (\X_{s,t})$. In addition, for any $u = (u_a \colon a\in A )\subset \hat{\cL} (A)$ and $N$-sqgrm $\Y$ we define
\begin{equation}\label{defn_non_local_transl}
\hat{\mathcal{T}}_u[\Y]_{s,t}
:=  \hat{T}_u^{M} (q\mathrm{MinExt}^{M}(\Y))
=\proj_{M, \omega}\hat{T}_u(q\mathrm{MinExt}(\Y)_{s,t})\,,
\end{equation}
where $M=N\cdot N'$ with $N'$ the smallest  integer such that $u_a\in\hat{\mathcal{L}}^{N'}_{\omega}(A)$ for all $a\in A$ and $\hat{T}_v^M :=\proj_{M,\omega} \hat{T}_u\mathfrak{i}^M$ with the embedding $\mathfrak{i}^M:\,T^M_{\omega}(A)\to T_{\omega}((A))$. These two operations have equivalently a dynamical reinterpretation like Theorem \ref{thm:DynRen}.
\begin{proposition}\label{prop_dyn_ren}
Given a sqgrp (sqgrm) $\X$ over $A$ with weight $\omega$ and $u = (u_a \colon a\in A )\subset \hat{\cL} ((A))$, the composition $\hat{T}_u (\X_t)$ $(\hat{T}_u (\X_{s,t}))$ is again a sqgrp (sqgrm) which coincides (up to increments)  with the solution of
\begin{equation}\label{eq:transl_qODE1}
\dot \Z_t = \Z_t \otimes  (\hat{T}_{u} \dot{\X}_{t,t} )\, , \qquad \Z_0 = \hat{T}_u \X_0\,.
\end{equation}
The same result applies to good sqgrms  when $u\subset \hat{\cL} (A)$. In addition, for any  $u = (u_a \colon a\in A )\subset \hat{\cL} (A)$ and $N$-sqgrm $\Y$  the path $t\to \hat{\mathcal{T}}_u[\Y]_{0,t}$ is a sqgrp coinciding  with the solution of
\begin{equation}\label{eq:transl_qODE2}
\dot \W_t = \W_t\otimes_{M, \omega} (\hat{T}_{u} \dot{\Y}_{t,t} ) \,, \qquad \W_0 = \1, 
\end{equation}
where $M\in \mathbb{N}$ is contained in the definition of $ \hat{\mathcal{T}}_u[\Y]$ in \eqref{defn_non_local_transl}.
\end{proposition}
 
\begin{proof}
Choosing the family $\hat{u}=(\Phi^*_H(u_a) \colon a\in A )$, we can apply Theorem \ref{thm:DynRen} to the weighted sgrp (sgrm) $\hat{\X}=\Phi^*_H \X $ and $\hat{\Y}=\Phi^*_H \Y $. Then the paths associated to ${T}_{\hat{u}}(\hat{\X})$ and $\mathcal{T}_{\hat{u}}[\hat{\Y}]$ solve respectively the equations
\[
\dot{\hat{\Z}}_t = \hat{\Z}_t \otimes  (T_{\hat{u}} \dot{\hat{\X}}_{t,t} )\, , \qquad \hat{\Z}_0 = T_{\hat{u}} \hat{\X}_0 \,;
\]
\[
\dot{\hat{\W}}_t = \hat{\W}_t\otimes_{M, \omega} (T_{\hat{u}} \dot{\hat{\Y}}_{t,t} ) \,, \qquad \hat{\W}_0 = \1\,.
\]
Using the identity \eqref{eq:transl_quasi} and the properties of $\Phi^*_H$, $\Psi^*_H$ of preserving the weighted length, we deduce that $\hat{T}_{u}(\X)=\Psi^*_H{T}_{\hat{u}}(\hat{\X})$ and $\hat{\mathcal{T}}_u[\Y]=\Psi^*_H\mathcal{T}_{\hat{u}}[\hat{\Y}]$. Looking at the equation of the paths $\Psi^*_H \hat{\Z}$ and $\Psi^*_H \hat{\W}$, we obtain the equations \eqref{eq:transl_qODE1} and \eqref{eq:transl_qODE2}.
\end{proof}
We conclude the section  with a summary of the properties of translation operators $\hat{T}_{u} $ together with the differential equation \eqref{def:RDE}. For any collection of Lie polynomials with values in $A$ given by $v  = (v_a \colon a\in A )\subset \cL (A)$ and any family of vector fields $\left( f_a\colon a\in A \right)$, we consider the collection of vector fields $\left( f^v_a\colon a\in A \right)$ defined on any  $a\in A$ as 
\[ f^v_a:=f_a+ f_{v_a}\,, \]
where $f\colon\cL(A)\to\Vect^\infty(\bR^e)$ is the  Lie algebra morphism defined  in \eqref{def_vector_field} on a generic alphabet $A$. Extending $\left( f^v_a\colon a\in A \right)$  to all $T_{\omega}(A)$, we obtain a map $f^v\colon T_{\omega}(A)\to\Vect^\infty(\bR^e)$ which satisfies the identity $f^v= f_{T_v} $ over $\cL(A)$, as it was already shown  in the proof of Theorem \ref{thm:renormalisation_RDE}, when $A=\{1,\cdots,  d\}$.
By replacing the directions of translations $v$ with $u= (u_a \colon a\in A )\subset \hat{\cL} (A)$ we can uniquely define a Lie algebra morphism $\hat{f}^u\colon\hat{\cL}(A)\to\Vect^\infty(\bR^e) $ such that on any   $a\in A$ one has as 
\[ \hat{f}^u_{\Psi^*_H(a)}:=\hat{f}_{\Psi^*_H(a)}+ \hat{f}_{u_a}\,. \]
Extending $\hat{f}^u$ to all $T_{\omega}(A)$ and applying the definition of $\hat{f}$ with formula \eqref{eq:transl_quasi}, we also obtain a map  $\hat{f}^u\colon T_{\omega}(A)\to\Vect^\infty(\bR^e)$. Both maps $\hat{f}^u$ and $f^v$ allow to write the effect of translation at the level of differential equations.
\begin{theorem}\label{thm:final}
Let  $\X$ be a good sqgrm and $\W$ a $N$-sqgrm. For any given  $u=\{u_a\colon a\in A\}\subset \hat{\cL}(A)$ a path $Y\colon[0,T]\to\mathbb{R}^e $ solves one of the equation
$$
dY = \hat{f} (Y) d (\hat{T}_u(\X))\,,  \quad dY = \hat{f} (Y) d (\hat{\mathcal{T}}_u[\W])\,;
$$
if and only if it solves respectively
\begin{equation}\label{equation_1}
dY = \hat{f}^{u} (Y) d \X\,,  \quad dY =\hat{f}^{u} (Y) d \W\,.
\end{equation}
Alternatively,  by setting  $\hat{u}=\{\Phi^*_H(u_a)\colon a\in A\}\subset \cL(A)$ and $\hat{\X}= \Phi^*_H\X$, $\hat{\W}= \Phi^*_H\W$ then $Y$ solves also equivalently
\begin{equation}\label{equation_2}
dY = f^{\hat{u}} (Y) d \hat{\X}\,,  \quad dY = f^{\hat{u}} (Y) d \hat{\W}\,.
\end{equation}
\end{theorem}

\begin{proof}
Equivalences \eqref{equation_1} and \eqref{equation_2} follow by combining the proof of  Theorem \ref{thm:renormalisation_RDE} and Proposition \ref{lem:quasi_ONB}. The only thing to check is to describe the relations between the maps $\hat{f}^u$ and $f^{\hat{u}}$ and the translations $T_{\hat{u}}$ and $\hat{T}_{u}$. Indeed, it follows from the the definition of $\hat{T}_{u} $ and $\hat{f}^u$ that one has
\[\hat{f}^u= \hat{f}_{\hat{T}_{u}}\]
on any element $\Psi^*_H(a)$ and consequently over all $\hat{\cL}(A)$. Following the first argument in the  Theorem \ref{thm:renormalisation_RDE} we get the first equivalence \eqref{equation_1}. Using the explicit definition of $\hat{f}$ and identity \eqref{eq:transl_quasi}, we deduce also the following equality on $\hat{\cL}(A)$
\[ \hat{f}_{\hat{T}_{u}}= f_{\Phi^*_H\hat{T}_{u}}= f_{T_{\hat{u}}\Phi^*_H}= f^{\hat{u}}_{\Phi^*_H}\,.\]
Applying the same argument as in the proof of $(ii)$ in Proposition \ref{lem:quasi_ONB}, we conclude.
\end{proof}

We conclude the section by extending the  time translation of Corollary \ref{cor_time_transl} in the quasi-shuffle context. For these purposes, we suppose given a generic alphabet $A^1$ with commutative bracket $\{\,,\}_1$ and weight $\omega_1$. Then we simply add an extra time component as  new letter $\hat{0}$  and we define the extended alphabet $\bar{A}=\{\hat{0}\}\cup A^1$ together with weight $\omega\colon \bar{A}\to \mathbb{N}^*$ defined by extending $\omega_1$ and putting  $\omega(\hat{0})=1$ and  the extended commutative bracket $\{\,, \} $ given  for any $a,b\in A^1$
\[\{a,\hat{0}\}=\{\hat{0},a\}=\{\hat{0},\hat{0}\}= 0\,,\quad \{a,b\}=\{a,b\}\,.\]
Under these conditions the signature of the time component $\X^0\colon [0,T]^2\to G(\R)$
\[ \X^0_{s,t}:=\exp_{\otimes}((t-s)\hat{0})\] 
can be easily embedded in $\hat{G}(\bar{A})$ and for any good sqgrm   $\X^1:[0,T]^2 \to \hat{G} (A_1)$ we can introduce again
$$\bar{\X}:=(\X^0,\X^1)_{\text{min}}\,.$$
(see also Remark \ref{min_quasi}). Since $\Psi^*_H(\hat{0})= \hat{0}$, for any given a $u_0\in \hat{\mathcal{L}}(\bar{A})$, we choose the translation  $\hat{T}_{u_0}\colon T(A)\to T(A)$ such that $\hat{T}_{u_0}\hat{0}=\hat{0}+u_0$ and such that $\hat{T}_{u_0}(\Psi^*_H(a))=\Psi^*_H(a)$. Then one has trivially $\hat{T}_{u_0}= id$ over the subalgebra of $T_{\omega}(\bar{A}) $ isomorphic to $T_{\omega_1}(A^1)$ and we can describe the relations for equations driven by $\bar{X}$. The proof follows easily from  Corollary \ref{cor_time_transl}, Theorem \ref{thm:final} and Proposition \ref{lem:quasi_ONB}.

\begin{corollary}\label{cor:srplus_quasi_trans}
Let  $\X^1$ be a good sqgrm and $u_0\in\hat{\mathcal{L}}(A)$. Then the translation $t\mapsto \hat{T}_{u_0}\bar{\X}_{0,t}$ can be described as the sum $\mathbf{U}_0\srplus \bar{\X}$, where $\mathbf{U}_0$ is given by
\[(\mathbf{U}_0)_{s,t}=\exp_{\otimes}(u_0(t-s))\,.\]
Moreover, for any given family $\left( f_{\hat{0}}, f_a \colon a\in A^1 \right)$  we define $\bar{f}\colon T(A)\to \Vect^\infty(\bR^e)$ and $f\colon T(A_1)\to \Vect^\infty(\bR^e)$ like in \eqref{map_f} starting  from $\left( f_{\hat{0}}, f_a \colon a\in A^1 \right)$ and $\left( f_a \colon a\in A^1 \right)$ respectively and using the notiations $\hat{u_0}= \Phi^*_Hu_0$, $\hat{\X}^1= \Phi^*_H\X^1$ equation 
$$dY = \hat{\bar{f}} (Y) d (\hat{T}_{u_{0}}(\bar{\X}))$$
is equivalent to 
$$
dY = (f_{\hat{0}} + \hat{\bar{f}}_{u_0})(Y) dt+ \hat{f} (Y) d \X^1\,,\qquad  dY = (f_{\hat{0}} + \bar{f}_{\hat{u_0}})(Y) dt+ f (Y) d \hat{\X}^1\,.
$$
\end{corollary}
\begin{remark}
Note that we recover in this smooth quasi-geometric rough path setting a property like the final statement of \cite[Thm. 30 (ii)]{Bruned2019}:
The smooth rough path increment/model $T_{v_0}\Z_t$ does not depend on the precise choice of the Hopf algebra homomorphism $T_{u_0}$, as long as $T_{u_0} \hat{0}=\hat{0}+u_0$ and $T_{u_0}=\id$ on $T_{\omega_1}(A^1)$.
\end{remark}

\section{Recast in an abstract Hopf algebra framework} \label{Hopf_algebra_section}

Many of the above constructions can be formulated in the abstract framework of Hopf algebras following the approach adopted in \cite{Tapia2018} and \cite{Manchon2018}. In this section, we introduce the notion smooth rough paths over a Hopf algebra and we will discuss their renormalisation. An interesting operation arises at this level: the canonical sum of rough paths, which is applied to present an alternative construction to renormalisation of branched rough paths,  along the lines of \cite{Bruned2019}.

\subsection{Smooth rough paths on a Hopf algebra}
In what follows we fix $(\mathcal {H}, \cdot , \Delta)$  a connected, $\mathbb{N}$-graded and locally finite  commutative Hopf algebra over $\mathbb{R}$. That is to say there is   a sequence of finite-dimensional vector spaces $\{\mathcal{H}_{n}\}_{n\geq 0}$ such that
$$
\mathcal{H}=\bigoplus_{n=0}^\infty {\mathcal H}_{n}\,, \quad \mathcal{H}_{0}\approx \mathbb{R}= \langle \mathbf{1}\rangle\,,
$$
where  $\mathbf{1}$ is the unit of $ \mathcal{H}$. For any given Hopf algebra in this class, we introduce two different notions of dual space: \textbf{full dual} ${\mathcal H}^\prime$ and the \textbf{the graded dual} ${\mathcal H}^*$ respectively defined by
\[{\mathcal H}^\prime:={\rm Hom}\left({\mathcal H}, \mathbb R\right)={\rm Hom}\left(\bigoplus_{n=0}^\infty  {\mathcal H}_n, \mathbb R\right)\,,\quad\mathcal{H}^*:=\bigoplus_{n=0}^\infty {\mathcal H}_{n}^{\prime}= \bigoplus_{n=0}^\infty {\rm Hom}({\mathcal H}_{n}, \mathbb R) \,,\]
where $V^{\prime}={\rm Hom}(V, \mathbb R)$ stands for the space of continuous $\mathbb R$-valued linear  forms on a topological vector space $V$. There is   a canonical pairing  $\langle\cdot, \cdot\rangle$  between ${\mathcal H}^\prime$ and ${\mathcal H}$ and by equipping ${\mathcal H}^\prime$ is with the weak topology i.e., the weakest topology  which  the evaluation maps $v^*\mapsto \langle v^*, u\rangle,\, u\in {\mathcal H}$ the space $\mathcal{H}^*$ lies dense in ${\mathcal H}^\prime $, see \cite[Lem.~1.7]{Bogfjellmo2016},\cite{Bogfjellmo2018}.      
 
 Using the  grading of $\mathcal{H}$, we can also define for any $N\in \mathbb{N}$ the truncated spaces
 \[\left(\mathcal{H}^{*}\right)^N:=\bigoplus_{n=0}^N \mathcal{H}_{n}^*\,,\quad \mathcal{H}^N:=\bigoplus_{n=0}^N {\mathcal H}_{n}\,,\]
 which yield two natural filtrations  and $\{(\mathcal{H}^{*})^N\colon N\in \mathbb{N}\}$,  $\{\mathcal{H}^N\colon N\in \mathbb{N}\}$ for the vector spaces  $\mathcal{H}^*$ and $\mathcal{H}$. We note that $\left(\mathcal{H}^{*}\right)^N=\left(\mathcal{H}^{N}\right)^*$ and will henceforth write $\mathcal{H}^{*N}$.

 The coproduct  $\Delta$ on ${\mathcal H}$ induces   by duality   a product $\star$ in ${\mathcal H}^\prime$ and a fortiori on $\mathcal{H}^*$ defined on any  $\alpha\,, \beta\in \mathcal{H}^\prime$ and $h\in \mathcal{H}$ via the identity
 \[\langle \alpha\star \beta, h\rangle := \langle (\alpha \otimes \beta),\Delta h\rangle_2\,,\]
where $\langle \cdot, \cdot\rangle_2$ is the canonical pairing between ${\mathcal H}^\prime\otimes {\mathcal H}^\prime $ and $ \mathcal{H}\otimes {\mathcal H}$ induced by the canonical dual pairing $\langle\cdot, \cdot \rangle$.  Moreover it is also possible to define a coproduct $\Delta^*$ only on $\mathcal{H}^*$ from the identity
\[
\langle \Delta^* w,u\otimes v \rangle_{2}:= \langle w,u  v \rangle\,,
\]
obtaining the dual Hopf algebra $(\mathcal{H}^*, \star, \Delta^*)$. The coproduct $\Delta$ is compatible with the filtration and hence so is the product $\star$, which maps    $\mathcal{H}^{*M}\times \mathcal{H}^{*N}$  to $ \mathcal{H}^{*M+N}$. The projections $\pi_N: {\mathcal H}\longrightarrow {\mathcal H}/\oplus_{n=N+1}^\infty {\mathcal H}_n$ onto the quotient by the ideal $\oplus_{n=N+1}^\infty {\mathcal H}_n$ is an algebra morphism so that  $\star_N:=\pi_N \star $ restricted to ${\mathcal H}_N$ defines a  truncated product  $ \star_N: \mathcal{H}^{*N}\times \mathcal{H}^{*N}\to \mathcal{H}^{*N}$. Much in the same way as  rough paths over a Hopf algebra $\mathcal{H}$ were defined, see  \cite{Tapia2018,Manchon2018}, we now define smooth rough paths and models.

 \begin{definition}
 We call a level-$N$ {\bf  smooth $\mathcal{H}$ rough path over}  (in short: $N$-\textbf{s$\mathcal{H}$rp}) any non zero path $\X: [0,T] \to {\mathcal H}^{*N}$  satisfying the following properties:
 \begin{itemize}
\item[(a'' i)] for all times $t \in [0,T]$ and for all  $ h\in {\mathcal H}^K$,   $k\in  { \mathcal H}^L$,  with $K+L= N$ one has
\begin{equation}\label{eq:shuffle_hopf}
\langle\X_{t},v \cdot w\rangle = \langle \X_{t},v\rangle\,\langle\X_{t},w\rangle\,.
\end{equation}
\item[(a'' ii)] For every word $ h\in {\mathcal H}^N$, the map $t \mapsto \langle\X_{t}, h\rangle$ is smooth.
\end{itemize}
 We call level-$N$ {\bf smooth rough model over $\mathcal{H}$}  (in short: $N$-\textbf{s$\mathcal{H}$rm}) any map $\X: [0,T]^2 \to \mathcal{H}^{*N}$  which satisfies property \eqref{eq:shuffle_hopf}   for all $\X_{s,t}$ as well as the following  properties:
\begin{itemize}
\item[(b' i)] The following abstract Chen relation holds:
\begin{equation}\label{eq:Chensrel_H}
\X_{su} \star_N \X_{ut}=\X_{s,t}\,
\end{equation} 
for any $s,u,t\in[0,T]$. 
\item[(b' ii)] For every  $h \in \mathcal{H}^N$, the map $t \mapsto \langle\X_{s,t},w\rangle$
is smooth, for one (equivalently: all) base point(s) $s \in [0,T]$.
\end{itemize}
By {\bf smooth rough path (model)} (in short: \textbf{s$\mathcal{H}$rp} and \textbf{s$\mathcal{H}$rm}), we mean a path (map) with values in  $\mathcal{H}^\prime$  for which  \eqref{eq:shuffle_hopf} holds  for any $w,v\in \mathcal{H}$ and relation \eqref{eq:Chensrel_H} holds with $\star$. 
\end{definition}
Algebraic properties of smooth rough paths on $\mathcal{H}$ are encoded in the specific Lie group (Lie algebra) structures of $\mathcal{H}'$. Following  \cite[Defn.~2.6]{Tapia2018},   property  \eqref{eq:shuffle_hopf} (in the case when $v,w\in \mathcal{H}$), amounts to   $\X_t$ belonging to the \textbf{group of  $N$-truncated characters} $G^N(\mathcal H)$ or the \textbf{group of characters} $G(\mathcal H)$,  for any $t$. The $\left(G({\mathcal H}), \star\right) $ is a topological Lie group when equipped with the topology of pointwise convergence, see \cite{Bogfjellmo2016} and by quotienting $(G^N(\mathcal H), \star_{N})$ is  a finite-dimensional Lie group. Their Lie algebras can be explicitly described by  the set ${\mathfrak g}^N({\mathcal H})$ of  \textbf{level-$N$ truncated  infinitesimal characters}, resp. the set ${\mathfrak g}({\mathcal H})$ of \textbf{infinitesimal characters} i.e,  of elements $\alpha\in {\mathcal H}^{\star N} $, resp. $\alpha\in {\mathcal H}^{\star} $  such that  for all  $ (h, k)\in \mathcal{H}^K\times  \mathcal{H}^L$ with $K+L= N$, resp.   $ (h, k)\in \mathcal{H}^2$,  the following property holds:
  \begin{equation}\label{eq:infshuffleH}
 \langle \alpha, h \cdot k\rangle=\langle \alpha, h \rangle\, \langle \mathbf{1}^*, k   \rangle+  \langle \mathbf{1}^*, h \rangle\,\langle \alpha,k\rangle\,,
 \end{equation} 
 where $\mathbf{1}^* $ is the counit of $\mathcal{H}$. The Lie brackets $[\cdot, \cdot]_{\star_N}$ on  ${\mathfrak g}^N({\mathcal H})$ and  $[\cdot, \cdot]_\star$  on ${\mathfrak g}({\mathcal H})$  are given by the commutators  of  $\star_N$  and $\star$. The pair $\left({\mathfrak g}({\mathcal H}),[\cdot, \cdot]_\star\right)$  defines a topological Lie algebra, see \cite{Bogfjellmo2018}. A special role is played by the set of \textbf{primitive elements} $P(\mathcal{H}^*)=  {\mathfrak g}({\mathcal H})\cap  \mathcal{H}^*$ which is a Lie algebra with the induced Lie brackets.
 To relate these Lie algebras with the the character groups we introduce the exponential and truncated exponential map $\exp_\star:\mathcal{H}^{\prime}\to \mathcal{H}^{\prime}$, $\exp_{\star_N}:\mathcal{H}^{*N}\to \mathcal{H}^{*N}$ defined by
 \begin{equation}\label{eq_exp_Hopf}
 \exp_\star {\bf x}:=\sum_{n=0}^\infty\frac{ {\bf x}^{\star n}}{n!}\,, \quad \exp_{\star_N}{\bf x}= \sum_{n\geq 0}^{N}\frac{\mathbf{x}^{\star n}}{n!}|_{\mathcal{H}^N}\,.
 \end{equation}
When restricted to the Lie algebras,  they induce bijections $\exp_{\star_N}: {\mathfrak g}^N({\mathcal H})\to G^N({\mathcal H})$ and the  $\exp_\star: {\mathfrak g}({\mathcal H})\to G({\mathcal H})$, turning $G({\mathcal H})$ into an {\bf analytic Lie group}, see  \cite[Thm.~3.7 and App.~B]{Bogfjellmo2016} and \cite[Thm.~3.9]{Bogfjellmo2018} for the properties of $\exp_\star$ and $G({\mathcal H})$. 

In practice, we can identify $N$-s$\mathcal{H}$rp $\X_t$ and $N$-s$\mathcal{H}$rm $\X_{s,t}$ have the same equivalence properties of their quasi geometric equivalent  by considering the increments of with respect to $\star$, see equation \eqref{eq:map_to_path}. Similarly to Definition \ref{def_extension}, we define the \textbf{extension  of a $N$-s$\mathcal{H}$rp}. The \textbf{diagonal derivative} 
\[
\dot{\X}_{s,s}:= \partial_t |_{t=s} \X_{s,t}= \X_s^{-1} \star \dot{\X}_s
\]
of  any given $\X$ in s$\mathcal{H}$rp lies in the Lie algebra $g(\mathcal{H})$. These properties allow to extend Theorem \ref{fundthmSGRP} on any smooth rough path over $\mathcal{H}$.
\begin{theorem}[Fundamental Theorem of s$\mathcal{H}$rp]\label{fundthmSGRPH}
Let $N \in \bN$. Any $\Y$ in $N$-s$\mathcal{H}$rp uniquely extends to some $\X$ in s$\mathcal{H}$rp which is minimal in the sense that 
$$
\X_s^{-1} \star \dot{\X}_s \in \mathfrak{g}^N (\mathcal{H})
$$
for all $s\in[0,T]$. We call $\mathcal{H}\mathrm{MinExt(\Y)} := \X$ the $\mathcal{H}$-{\bf minimal extension} of $\Y$ and also $\mathcal{H}\mathrm{MinExt^{N'}(Y)} := \pi_{N'} \X$, for $N' > N$, the $N'$-minimal extension of $\Y$. Moreover, for any $[s,t] \subset [0,T]$ the associated sgrm of $\X_{s,t}$  only depends on $\Y|_{[s,t]}$. We call this object the {\bf $\mathcal{H}$ signature} of $\Y$ on $[s,t]$, in symbols $\mathrm{Sig}_{\mathcal{H}}(\Y|_{[s,t]})$. A s$\mathcal{H}$rm $\X$ is called a {\bf good s$\mathcal{H}$rm} if it satisfies $\X= \mathcal{H}\mathrm{MinExt(\Y)}$ for some $N$-s$\mathcal{H}$rm $\Y$.
\end{theorem}
\begin{proof}
The proof of the result goes as that   of Theorem \ref{fundthmSGRP}, modulo the  replacement of  $\otimes$ by $\star$. The only properties to check in this generalised context is the existence and uniqueness of a smooth solution   for the initial value problem  \begin{equation}\label{eq:evolutioneq}
\dot{\gamma}(t)= \gamma(t)\star \eta(t) \,, \quad
\gamma(0)=\1^*
\end{equation}
for any given smooth curve $\eta:[0, T]\to {\mathfrak g}({\mathcal H})$. Moreover,  we need to prove that for any given couple $X, \bar{X}\in G^{N+1}(\mathcal{H})$ such that $\langle X,h\rangle = \langle X, h\rangle$ for any  $h\in \mathcal{H}_N$ then $X-\bar{X}$ belongs to the center of  $G^{N+1}(\mathcal{H})$. The first property  follows from the regularity in the sense of Milnor, see \cite{Milnor}, of  the Lie group $G({\mathcal H})$, see \cite[Thm.~B]{BogfjellmoSchmeding} and references therein. The second property follows from \cite[Prop.~2.10]{Tapia2018}. 
\end{proof}

\begin{example} Setting $\mathcal{H}:=\left(T(\mathbb{R}^d),\shuffle,\Delta\right) $ and $\mathcal{H}:=\left(T_{\omega}(A),\qshuffle,\Delta\right)$ we recover respectively Theorems \ref{fundthmSGRP} and \ref{thm:fund_quasi_geo}. In this identification, the operation $\star$, depending on $\Delta$ is identified with the concatenation product $\otimes$ see \cite{reutenauer1993free}. As recalled in Theorem \ref{thm:hoffmann_iso}, the  Hoffman's exponential and logarithm  \eqref{defn_exp} yield an isomorphism of graded Hopf algebras $\left(T(\mathbb{R}^d),\shuffle,\Delta\right) \simeq \left(T(\mathbb{R}^d),\qshuffle,\Delta\right)$. Hence the basic properties of these maps follow from the general fact that for any given Hopf algebra isomorphism $ \Gamma\colon \mathcal{H}\to \mathcal{\mathcal{K}}$ among two Hopf algebras $\mathcal{H}$ and $\mathcal{K}$ with the properties listed at the beginning, the adjoint map $\Gamma^*\colon \mathcal{H}^*\to \mathcal{K}^*$ is also an isomorphism. We observe also that the scalar product $\langle \cdot, \cdot\rangle$ defined on on $T(\R^d)$ and $T_{\omega}(A)$ can be used to identify the graded dual  $\mathcal{H}^*$ with $\mathcal{H}$ by means of the Riesz lemma, so that it is not necessary to introduce the notion of graded dual in that context.

\end{example}

\begin{example}
Another relevant example  in the context of renormalisation of rough paths arising in \cite{Bruned2019} is the Butcher-Connes–Kreimer Hopf algebra  $\mathcal {H}_{BCK}(\mathbb{R}^d)$ consists of polynomials of rooted forests $\tau$ with nodes decorated by the  finite set $\{1\,, \cdots \,, d\}$ together with the empty forest $\mathbf{1}$. A forest $f$ is graded accordingly to $|f|$, the number of its nodes and the coproduct $\Delta$ is defined on each tree $\tau$ as
 \begin{equation} \label{eq:CKHA}
 \Delta(\tau)=\sum_c P_c(\tau)\otimes R_c(\tau),
 \end{equation} 
 where the sum is taken over  a specific set of admissible cuts over the tree. The result of each cut produces a polynomial of trees  $P_c(t)$ is and $R_c(t)$ a tree corresponding to the root. We call the  $N$-s$\mathcal{H}$rp or s$\mathcal{H}$rp in this context \textbf{level-$N$ smooth branched rough paths}  and \textbf{smooth branched rough paths} (in shorts \textbf{$N$-sbrp} and \textbf{sbrp}), see \cite{gub10}. The operation $\star$ induced by the coproduct \eqref{eq:CKHA} coincide with the so called Grossman–Larson algebra forests \cite{grossman_larson89} and $P(\mathcal {H}^*_{BCK}(\mathbb{R}^d))$ coincides with  the free vector space $\langle\mathfrak{T}_d\rangle$ generated by $\mathfrak{T}_d$, the set of dual trees  decorated by the  finite set $\{1\,, \cdots \,, d\}$ and $\mathfrak{g}(\mathcal {H}_{BCK}(\mathbb{R}^d))$ is the vector space  $\langle\mathfrak{T}'_d\rangle$ of tree series.
 \end{example}
In adequacy with the previously results, level-$N$ smooth rough models over $\mathcal{H}$ are a special case of $\gamma$-H\"older rough paths over a Hopf algebra introduced in
 \cite{Manchon2018}.
\begin{proposition}
 Every $N$-s$\mathcal{H}$rm $\X$ is a $1/N$-regular $N$-truncated
$\mathcal{H}$ rough path, see \cite[Defn.~4.3]{Manchon2018}  and its minimal extension $\tilde{\X}$ coincides with the lift of $\X$, as constructed in \cite[Thm.~4.4]{Manchon2018}.
\end{proposition}
\begin{remark}
The same considerations of  Remark \ref{rk-weakly} apply also to the class of $\gamma$-regular $N$-truncated
$\mathcal{H}$ rough paths. 
\end{remark}

 \subsection{Translation of smooth rough paths on a Hopf algebra}

We now extend to the framework of smooth $\mathcal{H}$-rough paths the notion of translation discussed in Theorems \ref{thm:DynRen} and \ref{thm:transl_quasi}. To specify in which  direction  we can perform a translation,  we choose a generic  finite-dimensional subspace $ \mathcal{D}\subset P(\mathcal{H}^*)$. Fixing  a basis $\{h_1,\cdots, h_e\}$ of  $\mathcal{D}$, we obtain a fixed set of directions to apply a translation. By assigning  an element of $P(\mathcal{H}^*)$ to   each direction,  we introduce an abstract notion of translation.

\begin{definition}\label{def:Mv} 
Given a family of primitive elements $ v=\{v_i\colon i=1,\cdots,e \}\subset \mathfrak{g}(\mathcal{H})$ we call a \textbf{translation map over $\mathcal{D}$} any continuous  Lie algebra homomorphism $M_v\colon \mathfrak{g}(\mathcal{H})\to \mathfrak{g}(\mathcal{H})$ such that $M_{v}h_i= h_i+ v_i$ for any $i=1\,, \cdots \,, e$.
\end{definition}
\begin{remark}
Using the topological properties of $\mathfrak{g}(\mathcal{H})$, if $ v=\{v_i\colon i=1,\cdots,e \}\subset P(\mathcal{H}^*)$  to define a continuous Lie algebra homomorphism $M_v$  it is sufficient to have a Lie homomorphism $M_v\colon P(\mathcal{H}^*)\to \mathfrak{g}(\mathcal{H})$ since  the Lie algebra $P({\mathcal H}^\star)$ is dense in ${\mathfrak g}({\mathcal H})$, see \cite[Rrk.~3.11]{Bogfjellmo2018}. Note that Definition \ref{def:Mv}  only gives a pointwise definition for a fixed $v$.  As in the cases studied in this paper, in practice, we need a map $v\to M_v$ with certain consistency properties like $M_v M_u=M_{v+M_v u}$, a condition we do not impose here.  More generally we hope to investigate in future work, what properties one should impose on $\mathcal{D}$ and $v$ in a general Hopf algebra framework. A first suggestion of such a set of axioms was very recently given in \cite[Defn.6, Defn.7]{Rahm21}.
\end{remark}

\begin{remark} The previous translation maps $T_v$ and $\hat{T}_{\hat{v}}$ in the geometric and quasi-geometric setting are specific examples of translation over two different subspaces of primitive elements, i.e. the  vector space $\mathbb{R}^d$ and the free vector space generated by $\mathfrak{A}= \{\Psi^*_H(a)\colon a\in A\}$. This is a very specific situation one uses  the specific structure of $P(\mathcal{H}^*)$ as a free Lie algebra, which ensures both the existence and the uniqueness of a translation map. However, it was    shown in \cite[Ex.9]{Bruned2019} that  one can construct two different translation maps over the the same vector space $\mathcal{D}$ when $\mathcal{H}=\mathcal {H}_{BCK}(\mathbb{R}^d)$. The Lie structure of $\mathfrak{g}(\mathcal{H})$  is therefore  not sufficient to determine a unique translation. Hence the idea of a definition which does not involve a uniqueness of the translation  in its formulation. 

\end{remark}
Once given $ v=\{v_i\colon i=1,\cdots,e \}\subset \mathfrak{g}(\mathcal{H})$ and a translation map $M_v$ over $\mathcal{D}$, we can actually uniquely extend $M_v$ to a continuous $\star$ morphism $M_v\colon \cH'\to \cH'$ \textbf{the full translation map} which we denote in the same way. The extension is purely algebraic and follows by standard Milnor-Moore theorem \cite{MilnorMoore}. Indeed any translation map defines uniquely a Lie algebra morphism $M_v\colon P(\mathcal{H}^*)\to \mathcal{H}'$ which is compatible with the product $\star$. Using the universal property of the  universal enveloping algebra $\mathcal{U}(P(\mathcal{H}^*))$ and the Milnor-Moore theorem, the map $M_v$  uniquely extends to a $\star$ morphism $M_v\colon \mathcal {H}^*\to  \mathcal{H}'$ and by density it can be defined over all $\mathcal{H}'$. The map $M_v$ to perform translation of s$\mathcal{H}$rms Theorem \ref{thm:DynRen} and Proposition \ref{prop_dyn_ren}.

\begin{theorem}\label{thm:transl_H}
Given a s$\mathcal{H}$rp (s$\mathcal{H}$rm) a family $ v=\{v_i\colon i=1,\cdots,e \}\subset \mathfrak{g}(\mathcal{H})$  and a translation map $M_v$ over $\mathcal{D}$ the composition $M_v (\X_t)$, $(M_v (\X_{s,t}))$ is again a s$\mathcal{H}$rp (s$\mathcal{H}$rm) which coincides (up to increments)  with the solution of
\begin{equation}\label{eq:translODE}
\dot \Z_t = \Z_t \star  (M_v \dot{\X}_{t,t} )\, , \qquad \Z_0 = M_v\X_0\,.
\end{equation}
The same result applies to good s$\mathcal{H}$rms when $ v\subset P(\mathcal{H}^*)$. Supposing also that for any $N\in \mathbb{N}$ there exists an integer $L\geq N$ depending on $v$ such that $M_v\colon \mathcal{H}^{*N}\to\mathcal{H}^{*L}$, then for any $N$-s$\mathcal{H}$rp $Y$ the unique solution to 
$$
\dot \W_t = \W_t \star_{L}  (M_{v} \dot{\Y}_{t,t} )\, , \qquad \W_0 = \1^*, 
$$
defines a $L$-s$\mathcal{H}$rm, given by $\W_{s,t} = \W_s^{-1}\otimes_M \W_t$, which we call $\mathcal{M}_v[\Y]$. Moreover, we have the explicit form
\begin{equation}\label{eq:not_local_transl_H}
\mathcal{M}_v[\Y]_{s,t}
=  {M}_v^{L} (\Y_{s,t}^{L})
=\pi_{L}{M}_v(\X_{s,t})
\end{equation}
with algebra endomorphism $M_v^L := \pi_{M} M_v\mathfrak{i}^M$ of $(\mathcal{H}^*M,\star_M)$, using the (linear) embedding $\mathfrak{i}^M:\,\mathcal{H}^*M\to \mathcal{H}'$, and $\Y^M=\mathcal{H}\mathrm{MinExt}^{M}(\Y)$.
\end{theorem}
\begin{proof}
The theorem is then concluded by checking that  $M_{v} \X$ and $\mathcal{M}_v[\Y]$ solves the equation \eqref{eq:translODE}  and \eqref{eq:not_local_transl_H} for any  s$\mathcal{H}$rp $\X$ and $N$-s$\mathcal{H}$rp $Y$. This last check follows like in the proof of Theorem  \ref{thm:DynRen} by replacing $\otimes$ with $\star$. 
\end{proof}

As direct application of Theorem \ref{thm:transl_H}, we present a characterisation of the renormalisation of branched rough paths introduced in \cite{Bruned2019} when $\mathcal{H}=\mathcal {H}_{BCK}(\mathbb{R}^d)$. In that case,  the authors provided the existence a translation map $M_{v}\colon \mathfrak{T}_d\to \mathfrak{T}_d$ over the subspace $\mathcal{D}$ generated by the dual trees $\{\bullet_i^*\colon i=1\,,\cdots,d\}$ by using a specific property of the Lie algebra $(\langle\mathfrak{T}_d\rangle, [\,,]_{\star})$, which we briefly sketch.

Recall that a (left) pre-Lie algebra operation is a vector space $V$ equipped with a bilinear map $ \curvearrowright :V\otimes V\to V$ whose associator 
\[ (x,y,z)=(x\curvearrowright y)\curvearrowright z-x\curvearrowright (y\curvearrowright z)\]
is  invariant under the exchange of the two variables $y$ and $z$, see \cite{Manch}. Given a pre-Lie algebra, one can construct a Lie bracket via its   anti-symmetrisation
\[ [x\,,y]_{\curvearrowright}= (x\curvearrowright y) - (y\curvearrowright x)\,.\]
It turns out that $\langle\mathfrak{T}_d\rangle$ admits an explicit pre-Lie algebra structure $\curvearrowright$ which as well as satisfying  $[\,,]_{\curvearrowright}=[\,,]_{\star}$, also  has the property that $(\langle\mathfrak{T}_d\rangle, \curvearrowright)$ is isomorphic to the \textbf{free Pre-Lie algebra} over $d$ elements, see \cite{ChapLiv}.

Thanks to this property for any given family of tree series $v= (v_1\,, \cdots\,, v_d )\subset \mathfrak{T}'_d$, it is then possible to fix a unique translation map $M_v\colon \langle \mathfrak{T'}_d\rangle  \to \langle \mathfrak{T'}_d\rangle$  which satisfies 
\begin{equation}\label{unique_M}
M_v\bullet_i^*= \bullet_i^*+ v_i\,, \quad M_v(x\curvearrowright y)= (M_vx)\curvearrowright (M_vy)\,,
\end{equation}
for any $i =1, \cdots, d$. It follows from the standard properties of the operation $\curvearrowright $ and the grossmann-Larson product on the grading on trees $|\cdot|$ that for any given $v= (v_1\,, \cdots\,, v_d )\subset \mathfrak{T}_d$ $M_v$ maps $\mathcal {H}^{*N}_{BCK}(\mathbb{R}^d)$ to $\mathcal{H}^{*L}_{BCK}(\mathbb{R}^d)$ where $L=N\cdot N'$ with $N'$  the smallest integer such that $|v_i|\leq N'$ for any $i =1, \cdots, d$. From Theorem \ref{thm:transl_H} we deduce the following corollary.
\begin{corollary}\label{last_cor}
Given a sbrp (sbrm) $\X$  and $v = (v_1,\cdots v_d)\subset \mathfrak{T}'_d$, the composition $M_v (\X_t)$ $( M_v (\X_{s,t}))$ with $M_v$ uniquely defined by \eqref{unique_M} is again a sqgrp (sqgrm) which coincides (up to increments)  with the solution of
\begin{equation}\label{eq:abstract1}
\dot \Z_t = \Z_t \otimes  (M_v \dot{\X}_{t,t} )\, , \qquad \Z_0 = M_v \X_0\,.
\end{equation}
The same result applies to good sqgrms  when $v\subset \mathfrak{T}_d$. Under the same restriction on $v$, for any $N$-sbrm $\Y$ the path $t\to \mathcal{M}_u[\Y]_{0,t}$ is a $L$-sbrp coinciding  with the solution of
\begin{equation}\label{eq:abstract2}
\dot \W_t = \W_t\otimes_{L} (M_v \dot{\Y}_{t,t} ) \,, \qquad \W_0 = \1^*, 
\end{equation}
where $L\in \mathbb{N}$ is the smallest integer such that $|v_i|\leq N'$ for any $i =1, \cdots, d$.
\end{corollary}

\begin{remark}
Conditions \eqref{unique_M} identify uniquely a full translation map $M_v$ whose explicit calculation on forest is not direct. In \cite{Bruned2019} the authors obtained a dual description of the dual  map $M^*_v\colon \mathcal{H}_{BCK}(\mathbb{R}^d)\to \mathcal{H}_{BCK}(\mathbb{R}^d)$ using coalgebraic tools related to extraction and contraction of trees but an explicit expression of $M_v$ is still unknown. Looking at equations \eqref{eq:abstract1}, \eqref{eq:abstract2} we realize that in case of smooth branched rough paths it is sufficient to compute $M_v$ only on trees,  which should slightly simplify the computations. Moreover, when  $M_v$ coincides with the addition in \eqref{eq:srplusv0}, then we obtain an explicit equation which does not involve any coalgebraic tool.
\end{remark}
\begin{remark}Concerning the question of defining translation maps for more general Hopf algebras $\cH$, the example of the free pre-Lie algebra points towards the general approach of using some 'free' algebraic structure on the Lie algebra $P(\mathcal{H}^*)$, if it is available, to define the translation map as a universal homomorphism of the free object. Ideas to this regard were proposed in \cite[Defn.5.2.1]{Preiss21} and \cite[Defn.8]{Rahm21}. In the latter, this basic idea is also applied to the case of the free post-Lie algebra taking the role of $P(\mathcal{H}^*)$.
\end{remark}

\subsection{Canonical sum and minimal coupling of smooth rough paths}
\label{sec:srplus}
We pass now to the notion of sum in the generalised context of s$\mathcal{H}$rms, extending Definition \ref{def:sum_and_scalar_multiplication} to a generic Hopf algebra as before.
\begin{definition}\label{def:sum_and_scalar_multiplication_hopf} For any fixed s$\mathcal{H}$rms $\X,\Y:[0,T]^2\to\mathcal{H}'$ let $t\mapsto\Z_t \in \mathcal{H}'$ be the Cartan development of $\dot\X_{s,s}+\dot\Y_{s,s}$, i.e. the unique solution to
$$
\dot\Z_t=  \Z_t\star \left(\dot\X_{t,t}+\dot\Y_{t,t}\right), \quad \Z_0 = \1^*.
$$
We then write $\Z := \X \srplus\Y$ for the associated sgrm and call it the \textbf{canonical sum} of $\X$ and $\Y$. For any $\lambda\in \mathbb{R}$ we define also the sgrm $\Z=\lambda\srscalar\X$ via the Cartan development of  $\lambda \dot\X_{s,s}$, we call it the \textbf{canonical scalar multiplication}.
\end{definition}

\begin{remark}\label{rem:srscalar}
This scalar multiplication yields a new  scaling device for  smooth rough paths which strongly differs from the well-known dilation $\delta_\lambda$, where the latter is   defined   for any $x\in\mathcal{H}_n$ by  $\langle\delta_\lambda\X_{s,t},x\rangle=\langle\X_{s,t},\lambda^n x\rangle=\lambda^n\langle\X_{s,t},x\rangle$. 
In contrast to the pointwise scaling  $\delta_\lambda$, this new scaling  is a dynamical way to scale a smooth rough path. However, we observe that $(0\srscalar\X)_{s,t}=\delta_0\X_{s,t}=\mathbf{1}^*$ and that  $\overleftarrow{\X}$ (the  backward\footnote{The backward path is given by $\overleftarrow{\X}_{s,t}=\X_{(T-s),(T-t)}$.} rough path), $\delta_{-1}\X$ and $(-1)\srscalar\X$ are pairwise distinct. We expect this to   have interesting applications to the theory of signatures of rough paths in the geometric setting.
\end{remark}
From the vector space structure of the Lie algebra ${\mathfrak g}(\mathcal H)$, we easily deduce  that the    set of all smooth $\mathcal{H}$ rough models forms itself a vector space when equipped  with the sum $\srplus$ and the scalar multiplication $\srscalar$, with the set of all good smooth $\mathcal{H}$ rough models forming a subspace. In particular, for any real numbers $\lambda_1, \lambda_2, \lambda$ we have
\begin{equation*}
\X\srplus\Y=\Y\srplus\X,\quad\lambda\srscalar(\X\srplus\Y)=(\lambda\srscalar\X)\srplus(\lambda\srscalar\Y),\quad (\lambda_1+\lambda_2)\srscalar\X=(\lambda_1\srscalar\X)\srplus(\lambda_2\srscalar\X).
\end{equation*}
This is very much in contrast to the spaces of $\gamma$-Hölder or bounded $p$-variation rough paths for $\gamma\leq 1/2$ and $p\geq 2$, where it is basic folklore by now that such a vector space structure does not exist in any meaningful sense.

It is then possible to characterise the sum $\srplus$ via the group operation $\star$ up to some small remainder. The following theorem is an extension  in a Hopf algebra and smooth case of \cite[Section~3.3.1~B]{lyons1998}, which was stated in the context of geometric $p$-variation rough paths.
\begin{proposition}\label{prop:minimal_sum}
For any fixed couple of s$\mathcal{H}$rms $\X,\Y$ a map $\Z\colon [0,T]^2\to G(\mathcal{H})$ coincides with $\X \srplus \Y$ if and only if $\Z$ satisfies $\Z_{s,t}=\Z_{s,u}\star \Z_{u,t}$ for $s,u,t\in[0,T]$ and one has for any $s\in [0,T]$
 \begin{equation}\label{eq:prop:sum2H}
\Z_{s,t}=\X_{s,t}\star \Y_{s,t}+R_{s,t}\,,
 \end{equation}
 for some $R_{s,t}\in \mathcal{H}'$ such that for all $x\in  \mathcal{H}$ one has  $\langle R_{s,t},x\rangle= o(|t-s|)$ as  $t\to s$. Moreover, we have the relations
 \begin{equation}\label{eq:prop:sumH}
 \X_{s,t}\star\Y_{s,t}
 =\Y_{s,t}\star\X_{s,t}+r_{s,t}
 =\X_{s,t}+\Y_{s,t}-\mathbf{1}^*+r'_{s,t}\,,
 \end{equation}
 for some $r_{s,t}$, $r'_{s,t} \in \mathcal{H}'$ such that for all $x\in \mathcal{H}$ one has $\langle r_{s,t},x\rangle,\langle r'_{s,t},x\rangle= o(|t-s|)$ as $t\to s$.
\end{proposition} 
 \begin{proof}
Let us start by formula \eqref{eq:prop:sum2H}. We fix $\Z:[0,T]^2\to G(\mathcal{H})$ any map with $\Z_{s,t}=\Z_{s,u}\star\Z_{u,t}$ and for all $x\in\mathcal{H}$ one has $\langle\Z_{s,t}-\X_{s,t}\star\Y_{s,t},x\rangle= o(|t-s|)$. Then, by  considering the path $t\to \Z_{0,t}$ and fixing $s\in[0,T]$, we use the continuity of the map $\mathcal{H}'\ni\mathbf{x}\mapsto \Z_{0,s}\star \mathbf{x}$, for the weak convergence  with respect to the duality pairing of $\mathcal{H}'$ with $\mathcal{H}$ to have the equalities
\begin{align*}
\lim_{t\to s}\frac{\Z_{0,t}-\Z_{0,s}}{t-s}&=\Z_{0s}\star\lim_{t\to s}\frac{\Z_{s,t}-\mathbf{1}^*}{t-s} =\Z_{0,s}\star\lim_{t\to s}\frac{\X_{s,t}\star\Y_{s,t}-\mathbf{1}^*}{t-s}\\&=\Z_{0,s}\star\partial_t|_{t=s}(\X_{s,t}\star\Y_{s,t})=\Z_{0s}\star(\dot\X_{s,s}+\dot\Y_{s,s})\,.
\end{align*}
Thus the path $s\mapsto \Z_{0,s}$ is differentiable with derivative $s\mapsto\Z_{0,s}\star(\dot\X_{s,s}+\dot\Y_{s,s})$ , implying $\dot\Z_{s,s}=\dot\X_{s,s}+\dot\Y_{s,s}$, i.e. $\Z=\X\srplus\Y$. On the other hand, assuming that $\Z$ is given by $\X\srplus\Y$,  for all $x\in\mathcal{H}$ there are remainders $\theta^x_{s,t}= o(|t-s|)$, ${\theta}^{\prime x}_{s,t}= o(|t-s|)$ such that
\begin{equation}\label{sum_1}
\begin{split}
 \langle\Z_{s,t},x\rangle&=\langle\dot\Z_{s,s},x\rangle (t-s)+\theta^x_{s,t}\\&=\langle\dot\X_{s,s}+\dot\Y_{s,s},x\rangle (t-s)+\theta^x_{s,t}\\&=\langle\X_{s,t}+\Y_{s,t}- \mathbf{1}^*,x\rangle+{\theta}^{\prime x}_{s,t}\,.
 \end{split}
\end{equation}
Therefore the equivalence  \eqref{eq:prop:sum2H} will follow from  equivalence \eqref{eq:prop:sumH}. Recall that for   $\X_{s,t},\Y_{s,t}$ in $G(\mathcal{H})$ and $x$ in $ {\mathcal H}$, we have $\langle\X_{s,t}\star\Y_{s,t},x\rangle  =\langle \X_{s,t}\mathbin{\hat{\otimes}}\Y_{s,t},\Delta x\rangle$, where $\hat{\otimes}$ is the completed external tensor product, and furthermore $\langle \X_{s,t}, \mathbf 1\rangle= \langle \Y_{s,t}, \mathbf 1\rangle=1$.  Hence, 
\begin{equation}\label{sum_2}
    \langle\X_{s,t}\star\Y_{s,t},\mathbf{1}\rangle
    =\langle\Y_{s,t}\star\X_{s,t},\mathbf{1}\rangle
    =\langle\Y_{s,t}+\X_{s,t}-\mathbf{1}^*,\mathbf{1}\rangle
    =1.
\end{equation}
Furthermore, for every $x\in\mathcal{H}^{\geq 1}=\sum_{n=1}^\infty\mathcal{H}_n$, i.e.\ $\langle\mathbf{1}^*,x\rangle=0$, we write $\Delta x=\mathbf{1}\otimes x+x\otimes\mathbf{1}+\sum_{i=1}^n y_i\otimes z_i$ for some $n$ and $y_i,z_i\in\mathcal{H}^{\geq 1}$, thereby obtaining 
\begin{align*}
    \langle\X_{s,t}\star\Y_{s,t},x\rangle
    &=\langle \X_{s,t}\mathbin{\hat{\otimes}}\Y_{s,t},\Delta x\rangle
    =\langle\X_{s,t},x\rangle+\langle\Y_{s,t},x\rangle+\sum_{i=1}^n\langle\X_{s,t},y_i\rangle\langle\Y_{s,t},z_i\rangle\\
    &=\langle\X_{s,t}+\Y_{s,t}-\mathbf{1}^*,x\rangle
    +\underbrace{\sum_{i=1}^n\langle\X_{s,t},y_i\rangle\langle\Y_{s,t},z_i\rangle}_{o(|t-s|)}\, .
\end{align*}
The fact that this last  term is of order $o(|t-s|)$ follows from the  smoothness of $\X$ and $\Y$ and $\langle\X_{t,t},y_i\rangle=\langle\Y_{t,t},z_i\rangle=0$, which  for any index $i$, actually  yields the existence of a constant   $C_i>0$ such that $|\langle\X_{s,t},y_i\rangle\langle\Y_{s,t},z_i\rangle|\leq C_i(t-s)^2$ for all $t,s\in[0,T]$. 
Similarly, one shows that $\langle\Y_{s,t}\star\X_{s,t}-\X_{s,t}-\Y_{s,t}+\mathbf{1}^*,x\rangle= o(|t-s|)$.
\end{proof}

\begin{remark}
As pointed out in Remark \ref{rem:srplus_honest_geom_rp} with reference to \cite[Section~3.3.1~B]{lyons1998} for smooth geometric rough paths, we strongly conjecture that one can add the minimal extension of a $N$-s$\mathcal{H}$rm $\X$ to any general $\gamma$-Hölder $\mathcal{H}$ rough path $\mathbf{W}$ for any $\gamma\in(0,1)$, via the requirement $\langle(\X\srplus\mathbf{W})_{s,t},x\rangle=\langle\X_{s,t}\star\mathbf{W}_{s,t},x\rangle+o(|t-s|)$ using the sewing lemma.
\end{remark}

The canonical sum can now be used to define a \textbf{minimal coupling} of smooth rough paths in the situation where we have connected graded commutative Hopf algebras $(\mathfrak{H}^i)_{i=1,\dots,m}$ of the form
\[\mathfrak{H}^i= \bigoplus_{j=0}^\infty\mathfrak{H}_j^i\]
together with embeddings $\iota_i:(\mathfrak{H}^i)^*\to(\mathcal{H})^*$ which are injective graded Hopf algebra homorphisms, such that the sum of the non-unital parts (the kernels of the counit of $\mathcal{H}^*$) of the images $\hat{\mathfrak{H}}^i:=\iota_i(\mathfrak{H}^i)^*$ and $\hat{\mathfrak{H}}^i_j:=\iota_i(\mathfrak{H}^i)^*_j$ inside $\mathcal{H}^*$ forms a direct product, i.e.
\begin{equation*}
\sum_{i=1}^m \hat{\mathfrak{H}}^{i,\geq 1}=\bigoplus_{i=1}^m\hat{\mathfrak{H}}^{i,\geq 1}\, ,
\end{equation*}
where $\hat{\mathfrak{H}}^{i,\geq 1}
:=\bigoplus_{j=1}^\infty\hat{\mathfrak{H}}_j^i=\{x\in\hat{\mathfrak{H}}^i\colon \langle x,\mathbf{1}\rangle=0\}$,
and such that the first level of $\mathcal{H}^*=\oplus_{n=0}^\infty \mathcal{H}^*_n$ is actually spanned by the first levels of the $\hat{\mathfrak{H}}^i$, i.e.
\begin{equation*}
    \mathcal{H}^*_1=\bigoplus_{i=1}^m\hat{\mathfrak{H}}^i_1.
\end{equation*}
Any s$\mathfrak{H}^i$rp  $\X^i:[0,T]\to G(\mathfrak{H}^i)$ is trivially mapped via the embedding $\iota_i$ and  extended by  continuity inside $\mathcal{H}'$ to a s$\mathcal{H}$rp $\iota_i \X^i:[0,T]\to G(\mathcal{H})$. We define then the minimal coupling of the $\X^i$ as 
\begin{equation*}
(\X^1,\dots,\X^m)_{\text{min}}:=\iota_1 \X^1\srplus\cdots\srplus \iota_m\X^m.
\end{equation*}

As a special case,  we  consider two connected graded commutative Hopf algebras $\mathfrak{H}^0$, $\mathfrak{H}^1$ with embedded graded dual spaces $\hat{\mathfrak{H}}^0=\iota_0(\mathfrak{H}^0)^*$,  $ \hat{\mathfrak{H}}^1=\iota_1(\mathfrak{H}^1)^*$ such that 
\[\hat{\mathfrak{H}}^{0,\geq 1}+\hat{\mathfrak{H}}^{1,\geq 1}=\hat{\mathfrak{H}}^{0,\geq 1}\oplus\hat{\mathfrak{H}}^{1,\geq 1}\,, \quad  \mathcal{H}_1^*=\hat{\mathfrak{H}}^0_1\oplus\hat{\mathfrak{H}}^1_1\,.\]  
By taking $\hat{\mathfrak{H}}^0$ the unital subalgebra generated by just a single element $x_0\in(\mathcal{H})^*_1$ (which is automatically a graded sub Hopf algebra of $\mathcal{H}^*$), the only smooth rough models $\X^0:[0,T]\to G(\mathfrak{H}^0)$ are of the form
\begin{equation*}
\iota_0\dot\X^0_{s,s}=\psi(s)x_0
\end{equation*}
for some smooth $\psi\colon [0,T]\to\mathbb{R}$. Then one easily checks    that  
\begin{equation*}
  \iota_0\X^0_{s,t}=\exp_{\star}\left(\left(\int_s^t\psi(u)du\right) x_0\right)\,.
\end{equation*}
By taking $\psi=1$,  we find that for any s$\mathfrak{H}^1$rp  $\X^1:[0,T]\to G(\mathfrak{H}^1)$ the minimal coupling ${\bar{\X}}:=(\X^0,\X^1)_{\text{min}}$ is uniquely determined by
\begin{equation*}
\dot{\bar{\X}}_{s,s}=x_0+\dot\X^1_{s,s}.
\end{equation*}
In this situation, we can directly reinterpret $\bar{\X}$ as a specific type of full translation.
\begin{corollary}\label{cor_time_transl_hopf}
Let  $\X^1$ be a good s$\mathcal{H}$rm, $v_0\in \mathfrak{g}(\mathcal{H})$ and $M_{v_0}\colon \mathcal{H}^*\to\mathcal{H}^*$ a full translation map such that  $M_{v_0}x_0=x_0+v_0$ and $M_{v_0}$ restricted to $\iota_1(\mathfrak{H}^1)^*$ is the identity. Then the translation $t\mapsto M_{v_0}\bar{\X}_{0,t}$ can be described as the sum $\mathbf{V}_0\srplus \bar{\X}$, where $\mathbf{V}_0$ is given by
\[(\mathbf{V}_0)_{s,t}=\exp_{\star}(v_0(t-s))\,.\]
\end{corollary}
\begin{proof}
We obviously have
\begin{equation}\label{eq:srplusv0}
    \partial_t|_{t=s}(M_{v_0}{\bar{\X}}_{s,t})=M_{v_0}\dot{\bar{\X}}_{s,s}=x_0+v_0+\dot\X^1_{s,s}=v_0+\dot{\bar{\X}}_{s,s},
\end{equation}
which shows $M_{v_0}{\bar{\X}}=\mathbf{V}_0\srplus \bar{\X}$. 
\end{proof}

\begin{remark}
In this  general connected graded commutative Hopf algebra setting, we once again recover   a property similar to  the final statement of \cite[Thm. 30 (ii)]{Bruned2019}: the s$\mathcal{H}$rp  $M_{v_0}{\bar{\X}}_t$ does not depend on the precise choice of the Hopf algebra homomorphism $M_{v_0}$, as long as $M_{v_0} x_0=x_0+v_0$, $M_{v_0}|_{\hat{\mathfrak{H}}^1}=\id$.
\end{remark}

\bibliographystyle{alpha}
\bibliography{bibliography}

\end{document}